\documentclass[a4paper,12pt]{amsart}

\usepackage{fouriernc}
\usepackage[T1]{fontenc}

\usepackage{graphicx}
\usepackage{amsmath}
\usepackage{amssymb}
\usepackage[active]{srcltx}
\usepackage{hyperref}
\usepackage{bbm}
\usepackage{enumerate}
\usepackage{tikz}
\usepackage{comment}

\usepackage{geometry}
\usepackage{color}
\definecolor{red}{rgb}{1,0,0}
\def\red{\color{red}}
\definecolor{green}{rgb}{0.0,0.8,0.0}
\def\gree{\color{green}}
\definecolor{blue}{rgb}{0,0,1}
\def\blu{\color{blue}}
\definecolor{black}{rgb}{0,0,0}

\definecolor{grey}{rgb}{0.333,0.333,0.333}

\definecolor{lightgrey}{rgb}{0.666,0.666,0.666}

\definecolor{white}{rgb}{1,1,1}

\newtheorem{thm}{Theorem}%
\newtheorem{lem}{Lemma}
\newtheorem{prop}{Proposition}
\newtheorem{cor}{Corollary}
\newtheorem*{thm-non}{Theorem}
\theoremstyle{definition}
\newtheorem{defn}{Definition}[section]
\theoremstyle{remark}
\newtheorem{remark}{Remark}[section] %

\newtheoremstyle{customthm}
{}
{}
{}
{}
{\bfseries}
{.}
{.5em}
{}
\theoremstyle{customthm}
\providecommand{\customgenericname}{}
\newtheorem{innercustomgeneric}{\customgenericname}
\newcommand{\newcustomtheorem}[2]{
  \newenvironment{#1}[1]
  {
    \renewcommand\customgenericname{#2}
    \renewcommand\theinnercustomgeneric{##1}
    \innercustomgeneric
  }
  {\endinnercustomgeneric}
}
\newcustomtheorem{customthm}{Theorem}

\numberwithin{equation}{section}

\def\BB{{\mathbb B}}

\def\RR{{\mathbb R}}
\def\R{{\mathbb R}}

\def\ZZ{{\mathbb Z}}

\def\hatZZ{\widehat{\mathbb Z}}

\def\one{{\mathbbm{1}}}

\def\S{\operatorname{S{}}}

\def\scrA{{\mathcal A}}
\def\scrB{{\mathcal B}}
\def\scrC{{\mathcal C}}
\def\scrD{{\mathcal D}}
\def\scrE{{\mathcal E}}
\def\scrF{{\mathcal F}}
\def\scrG{{\mathcal G}}

\def\scrH{{\mathcal H}}

\def\scrK{{\mathcal K}}
\def\scrL{{\mathcal L}}
\def\scrM{{\mathcal M}}
\def\scrN{{\mathcal N}}

\def\scrP{{\mathcal P}}
\def\scrR{{\mathcal R}}

\def\scrU{{\mathcal U}}

\def\scrW{{\mathcal W}}
\def\scrX{{\mathcal X}}
\def\scrY{{\mathcal Y}}

\def\e{\mathrm{e}}

\def\diag{\operatorname{diag}}

\def\B{\operatorname{B{}}}

\def\L{\operatorname{L{}}}

\def\GL{\operatorname{GL}}

\def\SL{\operatorname{SL}}
\def\ASL{\operatorname{ASL}}

\def\OOO{\operatorname{O{}}}

\def\supp{\operatorname{supp}}

\def\vol{\operatorname{vol}}

\def\GamG{\Gamma\backslash G}
\def\GamGG{\Gamma_0\backslash G_0}
\def\GamH{\Gamma_H\backslash H}

\def\trans{\,^\mathrm{t}\!}

\newcommand{\bs}{\backslash}
\newcommand{\ve}{\varepsilon}
\newcommand{\bD}{\bar{D}}
\newcommand{\bU}{\bar{U}}
\newcommand{\bu}{\bar{u}}
\newcommand{\tg}{\tilde{g}}

\newcommand{\ts}{\tilde{s}}
\newcommand{\tv}{\tilde{v}}

\newcommand{\tbeta}{\tilde{\beta}}

\newcommand{\tscrW}{\widetilde{\scrW}}

\title[Extreme events for unipotent actions]{Extreme events and impact statistics for unipotent actions on the space of lattices}

\author[Jens Marklof]{Jens Marklof}
\address{Jens Marklof, School of Mathematics, University of Bristol, Bristol BS8 1UG, U.K.\newline \rule[0ex]{0ex}{0ex} \hspace{8pt}{\tt j.marklof@bristol.ac.uk}}

\author[Andreas Str\"ombergsson]{Andreas Str\"ombergsson}
\address{Andreas Str\"ombergsson, Department of Mathematics, Uppsala University, Box 480, SE-75106, Uppsala, Sweden
\newline \rule[0ex]{0ex}{0ex} \hspace{8pt}{\tt astrombe@math.uu.se}}

\author[Shucheng Yu]{Shucheng Yu}
\address{Shucheng Yu, School of Mathematical Sciences, University of Science and Technology of China (USTC), 230026, Hefei, China
\newline \rule[0ex]{0ex}{0ex} \hspace{8pt}{\tt yusc@ustc.edu.cn}}

\date{13 October 2025}
\thanks{JM's research was supported by EPSRC grant EP/W007010/1.   
AS was supported by the Knut and Alice Wallenberg Foundation and also by the Swedish Research Council Grant 2023-03411.
SY was supported by the
National Key R\&D Program of China No. 2024YFA1015100.
MSC (2020): 37D40, 11B57}

\begin{document}

\begin{abstract}
This paper extends a recent extreme value law for horocycle flows on the space of two-dimensional lattices, due to Kirsebom and Mallahi-Karai, to the simplest examples of rank-$k$ unipotent actions on the space of $n$-dimensional lattices. We analyse the problem in terms of the hitting time and impact statistics for the unipotent action with respect to a shrinking surface of section, following the strategy of Pollicott and the first named author in the case of hyperbolic surfaces. 
If $k=n-1$, the limit law is given by directional statistics of Euclidean lattices, whilst for $k<n-1$ we observe new distributions for which we derive precise tail asymptotics. 
\end{abstract}

\maketitle

\tableofcontents

\section{Introduction}\label{secIntro}

Let $G=\SL(n,\RR)$, $\Gamma= \SL(n,\ZZ)$, $\scrX=\GamG$, and  let $\mu$ be the Haar measure on $G$, normalized so that it represents the unique right $G$-invariant probability measure on $\scrX$. 
It is explicitly given by
\begin{equation} \label{Minkowski}
d\mu(g) \frac{dt}t = \big(\zeta(2)\zeta(3)\cdots\zeta(n)\big)^{-1}\;\det (X)^{-n}
\prod_{i,j=1}^n dX_{ij},
\end{equation}
where $X=(X_{ij})=t^{1/n}g \in \GL^+(n,\RR)$ with $g\in G$, $t>0$. We will also use the notation $\mu_0$ for the right $G_0$-invariant probability measure on $\GamGG$, with $G_0=\SL(n-1,\RR)$ and $\Gamma_0=\SL(n-1,\ZZ)$.

We define the following rank-$k$ unipotent action $h$ on $\scrX$,
\begin{equation}\label{equ:unipotentact}
\RR^k\times \scrX \to\scrX, \qquad 
(s,x) \mapsto h_s(x) = x U(s), \qquad U(s)=\begin{pmatrix} 1_k & 0 & 0 \\ 0 & 1_m & 0 \\ -s & 0 & 1 \end{pmatrix} ,
\end{equation}
with $k+m+1=n$. In the case $m=0$, we interpret this as $U(s)=\begin{pmatrix} 1_{n-1} & 0 \\  -s & 1 \end{pmatrix}$. It would be interesting to see how the present study can be extended to more general unipotent actions. For actions that are conjugate to $U(s)$, see Remark \ref{rem:conj} below.

Let $d: G\times G\to \RR_{\geq 0}$ be a left $G$-invariant Riemannian metric on $G$, i.e., $d(h g, hg')=d(g,g')$ for all $h,g,g'\in G$. 
This metric on $G$ defines a Riemannian metric $d_\Gamma(\,\cdot\,,\,\cdot\,)$ on $\scrX$ via
\begin{equation}\label{m1}
	d_\Gamma(x,x'):=\min_{\gamma\in\Gamma} d(g,\gamma g')  \qquad (x=\Gamma g,\; x'=\Gamma g').
\end{equation}
This is well-defined since $d(\,\cdot\,,\,\cdot\,)$ is left $G$-invariant. 

Fix an arbitrary norm $\|\cdot\|$ on $\RR^n$.
Define
\begin{equation}\label{equ:alpha}
\alpha(x) = \max_{v \in\ZZ^n g\setminus\{0\}} \frac{1}{\| v\|} ,
\end{equation}
with $g$ such that $x=\Gamma g$. The log of this function gives a natural way of measuring the distance from the ``origin'' $o=\Gamma\simeq \ZZ^n$ to $x=\Gamma g\simeq \ZZ^n g$. In fact, there is $c>1$ such that for $x$ sufficiently far from $o$,
\begin{equation}\label{dGammalogalphacompare}
c^{-1} d_\Gamma(x,o) \leq  \log \alpha(x) \leq c d_\Gamma(x,o)  .
\end{equation}
This inequality is the metric variant of Mahler's compactness criterion.

Fix an arbitrary norm $|\cdot|$ on $\RR^k$. The dynamical logarithm law in this context asserts that, for $\mu$-almost every $x$,
\begin{equation}\label{uloglaw}
\limsup_{T\to\infty}\sup_{|s|\leq T}  \frac{\log\alpha(h_s(x))}{\log T}  = \frac{k}{n} .
\end{equation}
This was proved for $k=1$ by Athreya and Margulis \cite{AthreyaMargulis,AthreyaMargulis2}; see Remark \ref{rmk:loglaw} below for $k\geq 1$ and \cite{Kelmer12,Kelmer19,Yu17} for more general unipotent actions.

The following theorem gives the fluctuations around the limit in \eqref{uloglaw}. Note that the limit distribution below will depend on the choice of the norms $|\,\cdot\,|$ and $\|\,\cdot\,\|$.

\begin{thm}\label{mainthm}
Let $\lambda$ be a Borel probability measure that is absolutely continuous with respect to $\mu$. Then, for any $Y\in\RR$, 
\begin{equation}\label{maineq}
\lim_{T \to \infty} \lambda\left\{ x\in\scrX :  \sup_{|s|\leq T}  \log\alpha(h_s(x))  >  Y + \frac{k}{n}\, \log T  \right\} 
=   \int_Y^\infty \eta(y) dy ,
\end{equation}
where $\eta\in\L^1(\RR)$ is a probability density with tail bounds
\begin{equation}\label{asytail0}
\int_Y^\infty \eta(y) dy = \kappa\, e^{-n Y} + O(e^{-2n Y}),\qquad\text{as }\: Y\to \infty,
\end{equation}
\begin{equation}\label{asytail1}
\int_{-\infty}^Y \eta(y) dy \asymp e^{-\frac{n(n-1)}k|Y|}
\cdot\begin{cases}
1&\text{if }\:k\geq2
\\
|Y|^m&\text{if }\:k=1,
\end{cases}
\qquad\text{as }\: Y\to-\infty. 
\end{equation}
The constant $\kappa>0$ in \eqref{asytail0} is given by
\begin{align}\label{equ:kappa}
\kappa=\frac{\vol_{\RR^k}(\B_1^k)}{\zeta(n)}\int_{\mathrm{pr}(\BB^{n,+}_1)}y_n^k\,dy_{k+1}\cdots dy_n,
\end{align}
where for any  $X>0$, $\B_X^k\subset \RR^k$ and $\BB_X^n\subset \RR^n$ are the closed balls centered at the origin with radius $X$ with respect to $|\cdot|$ and $\|\cdot\|$ respectively, 
$$
\BB_1^{n,+}:=\left\{y\in \BB_1^n: y_n>0\right\},
$$ 
and $\mathrm{pr}: \RR^n\to \RR^{m+1}$ is the projection map sending $ (x_1,\cdots, x_n)$ to $(x_{k+1},\cdots, x_n)$. 
\end{thm}

\begin{remark}\label{rmk:loglaw}
We will explain in Section \ref{sec:exelog} that the logarithm law \eqref{uloglaw} follows, for all $k\geq 1$, from Theorem \ref{mainthm} by a standard probabilistic argument based on the Borel-Cantelli Lemma.
\end{remark}

\begin{remark}\label{rem:conj}
Theorem \ref{mainthm}  extends to conjugates of the above unipotent action $U(s)$. This is equivalent to changing the defining norm for the $\alpha$-function; see Section \ref{sec:conj} for more details.
\end{remark}

\begin{remark}
The asymptotic bounds \eqref{asytail0} and \eqref{asytail1} follow from an interpretation of the distribution $\int_Y^{\infty}\eta(y)dy$ as the probability of a random rank-$n$ lattice avoiding certain sets (depending on $Y$) in $\RR^n$.  For such probability functions there is a quite general upper bound due to Athreya and Margulis \cite[Theorem 2.2]{AthreyaMargulis}. When $m=0$ (so that $k=n-1$) their bound matches with \eqref{asytail1}, but when $m>0$, \eqref{asytail1} is sharper. 
\end{remark}


Kirsebom and Mallahi-Karai's strategy \cite{Kirsebom} for dimension $n=2$ and $\lambda=\mu$ generalises without much difficulty to establish the limit law \eqref{maineq} also in this higher rank setting. The aim of the present paper is to take a slight detour -- as in \cite{MP25} -- via the impact statistics of the unipotent action for a shrinking surface of section, which we will introduce in Section \ref{secSurface}. This is of interest in its own right and provides a natural geometric-dynamical interpretation of the extreme value law. 

We will see in Section \ref{hittingtimestatsSEC} that for $k=n-1$ the hitting time statistics has the same limit as the fine-scale statistics of Farey fractions \cite{Marklof13}, or equivalently the directional statistics of primitive Euclidean lattice points \cite{MS10}. The Farey statistics are in turn closely related to the distribution of (logs of) smallest denominators \cite{smalld1,smalld2}, and for $n=2$ the distribution  coincides (up to simple scaling) with the the extreme value law of \cite{Kirsebom,MP25}. For $k<n-1$ the hitting time statistics are new. 

The surface of section we will use here appeared first in \cite{Athreya12,Marklof10,MS10} in the special case $k=n-1$. Particularly relevant to the ideas presented here is Tseng's paper \cite{Tseng23} on the distribution of closed horospheres (again with $k=n-1$) in shrinking cuspidal neighbourhoods. Tseng's setting corresponds to sequences of measures $\lambda$ supported on expanding closed horospheres rather than a fixed, absolutely continuous measure. We explain in Section \ref{secMore} that Theorem \ref{mainthm} also holds for such singular sequences of measures as long as they equidistribute sufficiently rapidly under the diagonal action $\varphi_t$ (for $t\to-\infty$).

In Section \ref{impactstatSEC} we consider the joint hitting time and impact distribution, which yields the entry time statistics for cuspidal neighbourhoods (Section \ref{entrytimessec}) and finally the extreme value law (Section \ref{secX}). The tail asymptotics of the hitting time statistics are derived in Section \ref{tailSEC}.  

Our findings for the unipotent action $U(s)$ (and its conjugates) are in stark contrast with the setting of hyperbolic dynamical systems, where the hitting time statistics have Poissonian limits, and the extreme value laws are given by the Gumbel distribution; see \cite{DolgopyatFayad2020,DolgopyatFayadLiu2022,FFT,Pollicott2009} for details and and further references.

\section{Surface of section} \label{secSurface}

Consider the subgroups
\begin{equation}\label{Hdef}
	H = \left\{ g\in G :  (0,1)g =(0,1) \right\}= \left\{ \begin{pmatrix} A  & \trans b \\  0 & 1 \end{pmatrix} : A\in G_0,\; b\in\RR^{n-1} \right\}
\end{equation}
and
\begin{equation}
	\Gamma_H = \Gamma\cap H  = \left\{ \begin{pmatrix} \gamma & \trans m \\ 0 & 1 \end{pmatrix} : \gamma\in\Gamma_0,\; m\in\ZZ^{n-1} \right\} .
\end{equation}
Here the notation $(0,1)$ is a shorthand for the vector $(0,\cdots,0,1)\in \RR^n$.
Note that $H$ and $\Gamma_H$ are isomorphic to the affine special linear groups $\ASL(n-1,\RR)= G_0\ltimes\RR^{n-1}$ and $\ASL(n-1,\ZZ)=\Gamma_0\ltimes\ZZ^{n-1}$, respectively. We normalize the Haar measure $\mu_H$ of $H$ so that it becomes a probability measure on $\GamH$; explicitly:
\begin{equation} \label{siegel2}
d\mu_H(g) = d\mu_0(A)\, db, \qquad g=\begin{pmatrix} A & \trans b \\ 0 & 1 \end{pmatrix}.
\end{equation}
Given a measurable map $M:\RR^n\setminus\{0\} \to G$, $y \mapsto M(y)\in G$ such that $(0,1) M(y)=y$, the map
\begin{equation}\label{My0}
	H \times \RR^n\setminus\{0\} \to  G, \qquad 
	(h, y) \mapsto hM(y), 
\end{equation}
provides a parametrization of $G$, with the Haar measure given by
\begin{equation}\label{Hahaar}
	d\mu = \zeta(n)^{-1} d\mu_H \, dy ;
\end{equation}
see \cite{Marklof10} and \cite[Section 7]{MS10} for further details. 

For any $L>0$, the set
\begin{equation}
\scrH(L) = \Gamma\backslash\Gamma H \big\{ D(y_n) : y_n\in(0, L^{-1}] \big\}, \qquad 
D(y_n)= \begin{pmatrix} y_n^{-1/(n-1)} 1_{n-1} & \trans 0 \\  0 & y_n\end{pmatrix}, 
\end{equation}
is a closed embedded submanifold of $\scrX$ \cite[Lemma 2]{Marklof10} and is transversal to the full (i.e., rank $k=n-1$) horospherical action; see proof of Theorem 6 in \cite{Marklof10}. 

Let 
\begin{align}\label{Rwdef}
Q:=\left\{ R(w)= \begin{pmatrix} w_{m+1}^{-(m+1)/k} 1_k & 0 & 0 \\ 0  &  w_{m+1} 1_m & 0 \\  0 & w'  & w_{m+1} \end{pmatrix} :
\quad w=(w',w_{m+1})\in \RR^m\times \RR_{>0}\right\};
\end{align}
this is a closed subgroup of $G$.
The corresponding section $\scrH(\scrW,L)$ for the $\RR^k$ action with $k<n-1$, is defined, for a given bounded Borel subset $\scrW \subset \RR^m\times\RR_{>0}$ with boundary of zero Lebesgue measure and non-empty interior, by
\begin{equation}\label{sect}
\scrH(\scrW,L) = \Gamma\backslash\Gamma H \big\{ R(w) :  w \in L^{-1}\scrW \big\} .
\end{equation}
This section is transversal to the orbits of the unipotent action but, unlike in the full rank case, $\scrH(\scrW,L)$ will have self-intersections. 

A key feature of the section $\scrH(\scrW,L)$ is its simple scaling under the following diagonal right action on $\scrX$, 
\begin{equation}\label{equ:phit}
\varphi_t(x) = x\Phi(t),\qquad \Phi(t)=  \begin{pmatrix} \e^{(m+1) t} 1_k  & 0 & 0 \\ 0 & \e^{-k t} 1_m & 0 \\ 0 & 0 & \e^{-k t} 
\end{pmatrix} .
\end{equation}
We have the relations
\begin{equation}\label{commute}
U(s) \Phi(t)  = \Phi(t) U(\e^{n t}s),\qquad 
R(w) \Phi(t)  = R(\e^{-k t} w)  ,
\end{equation}
which in turn imply that
\begin{equation}\label{secscale}
\varphi_t \scrH(\scrW,L)  = \scrH(\scrW,L\e^{kt}) , \qquad \varphi_t \circ h_s = h_{s\e^{n t}} \circ \varphi_t.
\end{equation}

Let $\scrF_H$ be a fixed fundamental domain of the $\Gamma_H$ action on $H$.
We will parameterise the section $\scrH(\scrW,L)$ by the set 
\begin{equation}\label{SigmaWLdef}
\Sigma(\scrW,L) = \scrF_H \big\{ R(w) :  w \in L^{-1}\scrW\big\} \subset G,
\end{equation}
via the map
\begin{equation}
\pi: \Sigma(\scrW,L) \to \scrX, \qquad g \mapsto \Gamma g .
\end{equation}
We have  an $r$-fold self-intersection at the point $x\in\scrH(\scrW,L)$, if the pre-image $\pi^{-1}(\{x\})$ (for the projection map $\pi:g\mapsto\Gamma g$) has $r$ distinct points.

Any element in $G_+:=\left\{g=(g_{ij})\in G: g_{nn}>0\right\}$ can be written uniquely in the form $g=hR(w)U(s)$ with $h\in H$, $w\in \RR^m\times \RR_{>0}$ and $s\in \RR^k$.
In these coordinates, the Haar measure (restricted to $G_+$) is given by 
\begin{equation}\label{muasnuds}
d\mu(g) = d\nu(a) \, ds, \quad a\in HQ,\, s\in \RR^k,
\end{equation} 
where 
\begin{equation}\label{nuDEF}
d\nu(a) = \zeta(n)^{-1} d\mu_0(A)\,db\, w_{m+1}^k dw , \qquad a=\begin{pmatrix} A & \trans b \\ 0 & 1 \end{pmatrix}R(w).
\end{equation}
The measure $\nu$ defines a measure on the section $\Sigma(\scrW,L)$, and we have
\begin{equation}\label{nuLscaling}
\nu\left( \Sigma(\scrW,L) \right) = L^{-n} \zeta(n)^{-1} \int_\scrW w_{m+1}^k dw = L^{-n} \nu\left( \Sigma(\scrW,1) \right) .
\end{equation}

It will be convenient to also use the map 
\begin{equation}\label{paraL}
\pi_L: \Sigma(\scrW,1) \to \scrX, \qquad g \mapsto \varphi_{\log L^{1/k}} (\Gamma g) 
\end{equation}
to parametrise $\scrH(\scrW,L)$. As noted above for $\pi$, this map is not always one to one.

\section{Hitting time statistics}
\label{hittingtimestatsSEC}

For given initial data $x\in\scrX$, we are now interested in the sequence of hitting times $\xi_1(x,L),\xi_2(x,L),\dots$, i.e., the times $\xi_j(x,L)=s\in\RR^k$ for which $h_s(x)\in \scrH(\scrW,L)$. We order them with respect to the norm $|\cdot|$ we have fixed before:  
\begin{equation}\label{xijdefordering}
0\leq |\xi_1(x,L)| \leq |\xi_2(x,L)| \leq \cdots.
\end{equation}
In case of equalities in \eqref{xijdefordering}, the ordering is chosen in an arbitrary fashion.
We will use the convention that we list hitting times with the multiplicity of the self-intersection in $\scrH(\scrW, L)$. That is, $\#\{ j : \xi_j(x,L)=s\} =\#\pi_L^{-1}(\{h_s(x)\})$. With this convention, we have for the number of hitting times in some bounded test set $\scrA\subset\RR^k$,
\begin{equation}\label{U1}
\#\left\{ j : \xi_j(x,L)\in\scrA  \right\} = \#\left( \Gamma g \cap \Sigma(\scrW,L) U(-\scrA)  \right), 
\end{equation}
where $g\in G$ is any fixed representative such that $x=\Gamma g$.
Up to possible reordering in the case of equalities in \eqref{xijdefordering}, the commutation relations \eqref{secscale} yield
\begin{equation}\label{secscale2}
\xi_j(x,L) = \e^{nt} \xi_{j}(\varphi_{-t}(x),L \e^{-kt}).
\end{equation}
Indeed, keeping our choices for $\xi_j(x,1)$ ($x\in \scrX$, $j\geq 1$), for any $L>0$ we may redefine $\xi_j(x,L)$ by
$$
\xi_j(x,L) = L^{n/k} \xi_{j}(\varphi_{-(\log L)/k}(x), 1).
$$
One then verifies that $\xi_1(x,L),\xi_2(x,L),\ldots$ indeed give all the hitting times $s\in\RR^k$ for which $h_s(x)\in\scrH(\scrW,L)$, counted with the correct multiplicity, and so that both \eqref{xijdefordering} and \eqref{secscale2} hold. 




Denote by $\hatZZ^n$ the set of primitive lattice points in $\ZZ^n$,  i.e.\ the integer vectors $m$ with
$\gcd(m_1,\ldots,m_n)=1$. 
Using the bijection
\begin{equation}\label{bij}
	\Gamma_H\backslash\Gamma \to \hatZZ^n, \qquad \Gamma_H \gamma \mapsto ( 0,1) \gamma ,
\end{equation}
we can write \eqref{U1} as
\begin{equation}\label{U2}
\begin{split}
\#\left\{ j : \xi_j(x,L) \in\scrA  \right\} 
& = \#\left( \hatZZ^n g \cap (0,1) R(L^{-1} \scrW) U(-\scrA) \right)\\
& = \#\left( \hatZZ^n g \cap \scrC(\scrA, L^{-1}\scrW) \right),
\end{split}
\end{equation}
where for any $\tscrW\subset\R^m\times\R_{>0}$ we define
\begin{equation}\label{scrC}
\scrC(\scrA,\tscrW)
=\left\{ y \in\RR^n :  (y_1,\ldots,y_k)\in y_n\scrA,\, (y_{k+1},\ldots, y_n) \in \tscrW  \right\} .
\end{equation}

\begin{thm}\label{thm2}
Let $\lambda$ be a Borel probability measure on $\scrX$ that is absolutely continuous with respect to $\mu$,
let $\scrA=\scrA_1\times\cdots\times\scrA_r$ 
where $\scrA_1,\ldots,\scrA_r$ are Borel subsets of $\RR^k$, and let $N\in\ZZ_{\geq 0}^r$.
Then
\begin{equation}\label{thm2res}
\lim_{L\to\infty}
\lambda\left(\left\{ x\in\scrX : \#\left\{ j  : L^{-n/k} \xi_j(x,L)\in\scrA_i \right\}  = N_i\, \forall i\right\}\right)  
= \Psi^{(r)}_N(\scrA) ,
\end{equation}
where
\begin{equation}\label{LD019}
\begin{split}
\Psi^{(r)}_N(\scrA) 
&= \mu\left(\left\{ x\in\scrX : \#\left\{ j : \xi_j(x,1)\in\scrA_i \right\}  = N_i \, \forall i\right\}\right) \\
&=\mu\left(\left\{ \Gamma g \in\scrX : \#\left( \hatZZ^n g \cap \scrC(\scrA_i,\scrW) \right) = N_i\, \forall i \right\} \right).
\end{split}
\end{equation}
\end{thm}

\begin{proof}
For $L=\e^{kt}$, 
the scaling property \eqref{secscale2}, yields
\begin{equation}
\begin{split}
&\left\{ x\in\scrX : \#\left\{ j : L^{-n/k} \xi_j(x,L) \in\scrA_i \right\}  = N_i\, \forall i \right\} \\
& =\left\{ x\in\scrX : \#\left\{ j :  \xi_j(\varphi_{-t}(x),1) \in\scrA_i \right\}  = N_i\, \forall i \right\} \\
& =\varphi_t\scrE
\end{split}
\end{equation}
with
\begin{equation}\label{thm2pf3}
\scrE = \left\{ x\in\scrX : \#\left\{ j : \xi_j( x,1) \in\scrA_i  \right\}  = N_i\, \forall i\right\} =
\left\{ \Gamma g \in\scrX : \#\left( \hatZZ^n g \cap \scrC(\scrA_i,\scrW) \right)= N_i\, \forall i \right\} .
\end{equation}
Since $\lambda$ is absolutely continuous we have
$\lambda=\psi d\mu$ for some non-negative $\psi\in \L^1(\scrX,\mu)$ with $\int_{\scrX}\psi(x)\,d\mu(x)=1$.
If $\psi\in \L^2(\scrX,\mu)$ then
\begin{align}\label{thm2pf1}
\lambda(\varphi_t\scrE)=\int_{\scrX}(\chi_{\scrE}\circ\varphi_{-t})\,\psi\,d\mu
\to\int_{\scrX}\chi_{\scrE}\,d\mu\int_{\scrX}\psi\,d\mu=\mu(\scrE)
\qquad\text{as }\: t\to\infty,
\end{align} 
by (the $\L^2$ formulation of) the mixing property of $\varphi_t$.
But \eqref{thm2pf1}  in fact extends to $\psi\in \L^1(\scrX,\mu)$,
as follows by a standard approximation argument, using the fact that
$\L^2(\scrX,\mu)$ is dense in $\L^1(\scrX,\mu)$.
\end{proof}

Theorem \ref{thm2} has the following corollary for the first hitting time.

\begin{cor}\label{cor1}
For $X>0$, 
\begin{equation}\label{cor1res}
\lim_{L\to\infty}
\lambda\left(\left\{ x\in\scrX : L^{-n/k} |\xi_1(x,L)|>X \right\}\right)  =  \int_X^\infty \psi_0(r)\, dr,
\end{equation}
with the probability density $\psi_0\in\L^1(\RR_{>0})$ defined by
\begin{equation}\label{equ:probdens}
\int_X^\infty \psi_0(r)\, dr = \Psi_0^{(1)}(\B_X^k) ,  
\end{equation}
where $\B_X^k=\{s \in\RR^k : | s | \leq X \}$ is as before.
\end{cor}

\begin{proof}
This is immediate from Theorem \ref{thm2} for $r=1$, $N_1=0$, $\scrA=\B_X^k$.
\end{proof}

We will in the following drop the superscript and write $\Psi_0(\B_X^k) = \Psi_0^{(1)}(\B_X^k)$.

 \begin{remark}
An inspection of formula \eqref{LD019} in the case $k=n-1$ yields that the limit distribution coincides with directional statistics of 
primitive points of an $n$-dimensional lattice \cite{MS10}, 
which in turn are the same as the fine-scale statistics of 
the $k$-dimensional sequence of Farey points 
\cite{Marklof13}.

For $k<n-1$,
if $\scrW$ is of the special form $\scrW=\scrW'\times(0,1]$
with $\scrW'\subset\RR^m$,
then we obtain an interpretation of the limit distributions in Theorem \ref{thm2} in terms of directional statistics of Euclidean model sets as follows. For a Euclidean lattice $\scrL=\ZZ^n g_0$ define the set
\begin{equation}\label{PWL}
\scrP(\scrW,\scrL) = \left\{ \pi'(y) : y\in\scrL, \pi_\perp(y)\in\scrW \right\} ,
\end{equation}
where $\pi':\RR^n\to\RR^{k+1}$ and $\pi_\perp:\RR^n\to\RR^m$
are given by $\pi'(y)=(y_1,\ldots,y_k,y_n)$
and $\pi_\perp(y)=(y_{k+1},\ldots,y_{k+m})$.
For $\scrW=\scrW'\times(0,1]$ with $\scrW'$ bounded with $\vol_{\RR^k}(\partial\scrW')=0$ and non-empty interior, 
the model (or cut-and-project) set $\scrP(\scrW,\scrL)$ is a Delone set in $\RR^{k+1}$, 
i.e., it is uniformly discrete and relatively dense.
The limit law for the directional statistics of $\scrP(\scrW,\scrL)$ was proved in \cite[Appendix A]{MS14}. Comparing this with Theorem \ref{thm2}, we see that the limit distributions are almost identical, with the caveat that we need to replace in \eqref{PWL} the lattice $\scrL=\ZZ^n g_0$ by its primitive subset $\widehat\scrL=\hatZZ^n g_0$.
Thus the limit distributions are not the same, but can be related to each other via a classical inclusion-exclusion argument. Note that the visible points in model sets studied in \cite{MS15} is generally not the same as the cut-and-project set $\scrP(\scrW,\widehat\scrL)$ defined from the primitive subset of a lattice. The directional statistics of more general adelic model sets were analysed in \cite{ElBaz}.
\end{remark}

\section{More general measures}\label{secMore}

If we restrict the family of permitted test sets,
we can relax the condition that $\lambda$ is absolutely continuous with respect to $\mu$
in Theorem~\ref{thm2}. 

\begin{defn}\label{def:admi}
We say that a Borel probability measure $\lambda$ on $\scrX$ is 
\textit{admissible} if 
\begin{equation}\label{lconrep}
\lim_{t\to\infty} \lambda\left(\varphi_t\scrE\right)  = \mu\left(\scrE\right)  
\end{equation}
holds for every Borel subset $\scrE\subset\scrX$ with $\mu(\partial\scrE)=0$.
\end{defn}

\begin{remark}\label{rmk:port}
This notation was also introduced in \cite{MP25}. We note that  by the Portmanteau theorem, the following are equivalent:
\begin{enumerate}
\item
$\lambda$ is admissible; 
\item $\limsup_{t\to\infty} \lambda\left(\varphi_t\scrC\right)\leq \mu\left(\scrC\right)  
$
for every closed subset $\scrC\subset\scrX$;
\item
$\liminf_{t\to\infty} \lambda\left(\varphi_t\scrU\right)\geq \mu\left(\scrU\right)  
$
for every open subset $\scrU\subset\scrX$.
\end{enumerate}
\end{remark}
An example of an admissible measure is given by 
\begin{equation}\label{exadmmeasure1}
\lambda(\scrE) = \rho(\{ B \in\RR^{k\times (m+1)} : x_0 V(B) \in\scrE\}) ,
\end{equation}
for any fixed Borel probability measure $\rho$ on $\RR^{k\times (m+1)}$ which is absolutely continuous with respect to Lebesgue, any fixed $x_0\in\scrX$, and 
\begin{equation}
V(B) = 
\begin{pmatrix} 1_k & B \\ 0 & 1_{m+1} \end{pmatrix}.
\end{equation}
The key point here is that $V(B)$ parametrises the full stable horospherical subgroup for $\varphi_t$. This in turn implies (by Margulis' thickening argument) that $\varphi_{-t} \{ x_0 V(B) : B\in \scrK \}$ becomes uniformly distributed in $\scrX$ with respect to $\mu$, as $t\to\infty$, for any compact set $\scrK\subset\RR^{k\times (m+1)}$ with non-empty interior and boundary of Lebesgue measure zero.
We can in fact choose more singular $\rho$, for which the admissibility of $\lambda$ follows from -- or extends -- Ratner's measure classification theorem. For the most recent results see for example \cite{Shah24} (and references therein) in the case when $\rho$ is supported on submanifolds, and \cite{BenardHeZhang2024,Datta24,Khalil23,Khalil25} (and references therein) for fractal measures $\rho$.

\vspace{5pt}

The variant of Theorem \ref{thm2} looks as follows: 
\begin{customthm}{2'}\label{thm2var}
\textit{Let $\lambda$ be an admissible probability measure on $\scrX$,
let $\scrA=\scrA_1\times\cdots\times\scrA_r$ 
where $\scrA_1,\ldots,\scrA_r$ are bounded Borel subsets of $\RR^k$ with
$\vol_{\RR^k}(\partial \scrA_i)=0$, and let $N\in\ZZ_{\geq 0}^r$.
Then
\begin{equation}
\lim_{L\to\infty}
\lambda\left(\left\{ x\in\scrX : \#\left\{ j  : L^{-n/k} \xi_j(x,L)\in\scrA_i \right\}  = N_i\, \forall i\right\}\right)  
= \Psi^{(r)}_N(\scrA) ,
\end{equation}
with $\Psi^{(r)}_N(\scrA)$ as in \eqref{LD019}.}
\end{customthm}
\begin{proof}
For each $i$, since both $\scrA_i$ and $\scrW$ have boundary of Lebesgue measure zero,
also $\scrC(\scrA_i,\scrW)$ has boundary of Lebesgue measure zero.
Indeed, one verifies that for each $i$, 
\begin{align}\label{thm2ppf1}
\partial\,\scrC(\scrA_i,\scrW)\subset\{y_n=0\}
\:\bigcup\:\{y\in\RR^n : (y_1,\ldots,y_k)\in y_n\partial\scrA_i\}
\:\bigcup\:\{y\in\RR^n : (y_{k+1},\ldots,y_n)\in\partial\scrW\},
\end{align}
and it is immediate (using Fubini) that each of the three sets in the right hand side of
\eqref{thm2ppf1} has Lebesgue measure zero.

Let $\scrE\subset\scrX$ be the set defined in \eqref{thm2pf3}.
We claim that  $\mu(\partial\scrE)=0$. Hence since $\lambda$ is admissible,
$\lim_{t\to\infty} \lambda\left(\varphi_t\scrE\right)  = \mu\left(\scrE\right)$,
and so the proof of Theorem \ref{thm2} carries over to the present case. 
We now prove the claim by observing that
\begin{align}\label{equ:bozero}
\partial\scrE\subset \left\{\Gamma g\in \scrX: \hat{\ZZ}^ng\cap \bigcup_{i=1}^r\partial\,\scrC(\scrA_i,\scrW)\neq \emptyset\right\}.
\end{align}
Since $\vol_{\RR^n}(\partial\,\scrC(\scrA_i,\scrW))=0$, it follows from Siegel's volume formula  \cite{Siegel1945} that the set in the above right hand side is of $\mu$-measure zero, and thus $\mu(\partial\scrE)=0$. It thus remains to prove \eqref{equ:bozero}. Take $\Gamma g\in\partial\scrE$,
then there exists a sequence of elements $g_j\in G$ with $g_j\to g$ and $\Gamma g_j\in\scrE$ for all $j$,
and another sequence of elements $\tg_j\to g$ with $\Gamma \tg_j\notin\scrE$ for all $j$.
The last fact means that for each $j$ there is some $i$ such that
$\#\bigl(\hatZZ^n\tg_j\cap\scrC(\scrA_i,\scrW)\bigr)\neq N_i$.
By passing to a subsequence, we may assume that there is a 
\textit{fixed} $i\in\{1,\ldots,r\}$ such that
\begin{align}\label{thm2ppf2}
\#\bigl(\hatZZ^n\tg_j\cap\scrC(\scrA_i,\scrW)\bigr)\neq N_i\qquad\text{holds for each $j$.}
\end{align}
On the other hand we have
\begin{align}\label{thm2ppf3}
\#\bigl(\hatZZ^n g_j\cap\scrC(\scrA_i,\scrW)\bigr)=N_i\qquad\text{for each $j$.}
\end{align}
Note that $\scrC(\scrA_i,\scrW)$ is bounded, since $\scrW$
and $\scrA_i$ are bounded.
Since $g_j\to g$ and $\tg_j\to g$, there exists a compact subset
$\scrG\subset G$ with $\{g_j\}\cup\{\tg_j\}\subset\scrG$,
and now 
$\bigcup_{g\in\scrG}\scrC(\scrA_i,\scrW)g^{-1}$ is a bounded subset of $\RR^n$,
meaning that it intersects $\hatZZ^n$ in a finite subset $F$.
By construction, for every $v\in\hatZZ^n\setminus F$
we have $vg_j\notin\scrC(\scrA_i,\scrW)$
and $v\tg_j\notin\scrC(\scrA_i,\scrW)$ for all $j$.
Hence it follows from \eqref{thm2ppf2} and \eqref{thm2ppf3} that
\begin{align}
\#\bigl(F \tg_j\cap\scrC(\scrA_i,\scrW)\bigr)\neq N_i
\quad\text{and}\quad
\#\bigl(F g_j\cap\scrC(\scrA_i,\scrW)\bigr)=N_i,
\qquad\forall j.
\end{align}
Hence for each $j$ there is a vector $v\in F$ such that
one but not both of $v\tg_j$ and $vg_j$ belongs to
$\scrC(\scrA_i,\scrW)$.
By passing to a subsequence we may assume that there is a \textit{fixed}
$v\in F$ such that for each $j$,
one but not both of $v\tg_j$ and $vg_j$ belongs to
$\scrC(\scrA_i,\scrW)$.
Using $g_j\to g$ and $\tg_j\to g$,
this implies that $vg\in\partial\scrC(\scrA_i,\scrW)$, finishing the proof of \eqref{equ:bozero}, and thus also the theorem. 
\end{proof}

\begin{remark}
Corollary \ref{cor1} extends verbatim to arbitrary admissible measures $\lambda$,
since the ball $\scrA=\B_X^k$ is bounded and has boundary of Lebesgue measure zero.
\end{remark}

\begin{remark}
Theorem \ref{thm2var} (and all later statements in this paper) remain valid
if we allow the measure $\lambda$ to depend on $L$,
say $\lambda=\lambda_L$,
subject to the assumption that
\begin{equation}\label{lconLnew}
\lim_{t\to\infty}\lambda_{e^{kt}}(\varphi_t\scrE)=\mu(\scrE)
\end{equation}
for all Borel sets $\scrE$ with $\mu(\partial\scrE)=0$.
An example of a family of such measures $\lambda_L$ is given by 
\begin{equation}\label{lambdatexdef}
\lambda_L \left(\scrE\right) = \lambda\left(\varphi_{\tau(L)-\frac{1}{k}\log L}\scrE\right),
\end{equation}
with any $\tau:\RR_{> 0}\to\RR_{> 0}$ such that  $\tau(L)\to\infty$ as $L\to\infty$, and $\lambda$ admissible as defined in \eqref{exadmmeasure1}. For $k=n-1$, this choice of $\lambda_L$ corresponds to the family of measures studied by Tseng \cite{Tseng23}.
\end{remark}


\section{Impact statistics}
\label{impactstatSEC}

%
%
%
%
%

We will now consider the impact parameter at the $j$th hit, which we define as the pre-image
\begin{equation}\label{omegajxLdefrep}
\omega_j(x,L) \in \pi_L^{-1}( \{h_{\xi_j{{(x,L)}}}(x)\} ) \subset \Sigma(\scrW,1)
\end{equation}
of the actual location $h_{\xi_j{{(x,L)}}}(x)\in\scrX$ of the hit. Recall that the number of distinct pre-images corresponds exactly to the multiplicity of $\xi_j(x,L)$; thus we may 
choose the points $\omega_j(x,L)$ so that 
\begin{align}\label{omegajxLmultpprecise}
\pi_L^{-1}\bigl(\{h_s(x)\}\bigr)=\bigl\{\omega_j(x,L) : j\in \ZZ_{\geq 1},\: \xi_j(x,L)=s\bigr\},
\qquad\forall s\in\R^k.
\end{align}
We also take the points $\omega_j(x,L)$ to be ordered so that
\begin{align}\label{omegajxLscaling}
\omega_j(x,L)=\omega_j(\varphi_{-t}(x),Le^{-kt}).
\end{align}
To see that this is possible, choose the points $\omega_j(x,1)$ so that 
\eqref{omegajxLmultpprecise} holds for $L=1$;
then for general $L>0$ set $\omega_j(x,L):=\omega_j(\varphi_{-(\log L)/k}(x),1)$.
Using the scaling properties \eqref{commute} and \eqref{secscale2},
one now verifies that both
\eqref{omegajxLmultpprecise} and \eqref{omegajxLscaling} 
hold for all $L>0$, $x\in\scrX$ and $j\in\ZZ_{\geq1}$.

Theorem \ref{thm2var} generalises to the following.

\begin{thm}\label{thm3}
Let $\lambda$ be an admissible probability measure on $\scrX$.
For $\scrA=\scrA_1\times\cdots\times\scrA_r\subset(\RR^k)^r$ bounded with $\vol_{\RR^k}(\partial \scrA_i)=0$, and $\scrD=\scrD_1\times\cdots\times\scrD_r\subset\Sigma(\scrW,1)^r$ with each $\scrD_i$ having compact closure in $HQ$ and $\nu(\partial\scrD_i)=0$, $N\in\ZZ_{\geq 0}^r$,
\begin{equation}
\lim_{L\to\infty}
\lambda\left(\left\{ x\in\scrX : \#\left\{ j : L^{-n/k} \xi_j(x,L) \in\scrA_i ,  \; \omega_j(x,L)\in\scrD_i \right\}  = N_i\, \forall i\right\}\right) = \Psi^{(r)}_N(\scrA,\scrD) ,
\end{equation}
where 
\begin{equation}\label{LD01910}
\begin{split}
\Psi^{(r)}_N(\scrA,\scrD)
& =  \mu\left(\left\{ x\in\scrX : \#\left\{ j : \xi_j(x,1) \in\scrA_i ,  \; \omega_j(x,1)\in\scrD_i \right\}  = N_i\, \forall i\right\}\right)\\ 
& =\mu\left(\left\{ \Gamma g \in\scrX : \#\left( \Gamma g \cap  \scrD_i U(-\scrA_i) \right) = N_i\, \forall i \right\} \right) .
\end{split}
\end{equation}
Here $\partial \scrD_i$ is the boundary of $\scrD_i$ considered as a subset of $HQ$.
\end{thm}

\begin{proof}
This is similar to the proof of  Theorem \ref{thm2var}:
For $L=e^{kt}$, the scaling properties \eqref{secscale2} and \eqref{omegajxLscaling} yield
\begin{align}
\bigl\{ x\in\scrX : \#\bigl\{ j : L^{-n/k} \xi_j(x,L) \in\scrA_i ,  \; \omega_j(x,L)\in\scrD_i \bigr\}  = N_i\, \forall i\bigr\}
=\varphi_t\scrE
\end{align}
with
\begin{equation}
\begin{split}
\scrE & =\left\{ x\in\scrX : \#\left\{ j : \xi_j(x,1) \in\scrA_i ,  \; \omega_j(x,1)\in\scrD_i \right\}  = N_i \, \forall i\right\} \\ & =\left\{ \Gamma g \in\scrX : \#\left(\Gamma g\cap \scrD_i U(-\scrA_i) \right) = N_i\, \forall i \right\}  .
\end{split}
\end{equation}

We only need to verify $\mu(\partial\scrE)=0$ so that $\lim_{t\to\infty} \lambda\left(\varphi_t\scrE\right)  = \mu\left(\scrE\right)$.
Recall that
\begin{align*}
Q=\{R(w): w\in\RR^m\times\R_{>0}\}
\end{align*}
is a closed subgroup of $G$.
Moreover, the map
\begin{align}\label{thm3ppf1}
\langle a,s\rangle\mapsto a U(s)
\end{align}
is a smooth diffeomorphism from $(HQ)\times\RR^k$
onto the open subset $G_+=\bigl\{(g_{ij})\in G: g_{nn}>0\bigr\}$ of $G$. Thus we have for each $i$,
$$
\partial(\scrD_iU(-\scrA_i))\subset (\partial\scrD_i)U(-\scrA_i)\cup \scrD_i(-\partial\scrA_i),
$$
and so the assumptions $\vol_{\RR^k}(\partial \scrA_i)=0$ and $\nu(\partial\scrD_i)=0$ imply that 
$\mu\bigl(\partial(\scrD_iU(-\scrA_i))\bigr)=0$
for each $i$, cf. \eqref{muasnuds} and \eqref{nuDEF}.
Now by similar arguments as in the proof of Theorem \ref{thm2var}, using the boundeness of $\scrA_i$ one verifies that
\begin{align*}
\partial\scrE\subset \left\{\Gamma g\in \scrX: g\in  \bigcup_{i=1}^r\partial(\scrD_iU(-\scrA_i))\right\}.
\end{align*}
This proves  
$\mu(\partial\scrE)=0$, and also the theorem.
\end{proof}

We now turn to the generalisation of Corollary \ref{cor1}. For $\omega\in\Sigma(\scrW, L)$, we define \begin{equation}\label{RomegaLdef}
\scrR(\omega,L)=\left\{ \xi_1(x,L) : x\in\scrX,  \; h_{\xi_1(x,L)}{{(x)}} = \pi(\omega) \right\} .
\end{equation}
This represents the set of all first hits with impact location at ${{\pi}}(\omega)$ on the section $\scrH(\scrW,L)$. Since almost every $x\in\scrX$ will eventually hit $\scrH(\scrW,L)$,
we then have, for any Borel measurable function $f:\scrX\to\RR_{\geq 0}$, 
\begin{equation}\label{dissi}
\int_{\scrX} f d\mu = \int_{{\Sigma(\scrW, L)}} \bigg( \int_{\scrR(\omega,L)} f(h_{-s}\circ  \pi(\omega))\, ds\bigg)d\nu(\omega).
\end{equation}
Setting $f=1$ yields the formula
\begin{equation}
\int_{\Sigma(\scrW, { L})}  \vol_{\RR^k}(\scrR(\omega,L))\, d\nu(\omega) = 1 .
\end{equation}

The following extends Corollary \ref{cor1}.

\begin{thm}\label{cor1B}
Let $\lambda$ be an admissible probability measure on $\scrX$, let $\scrB$ be a Borel subset of $\scrW$ satisfying $\vol_{\RR^{m+1}}(\partial\scrB)=0$, and set $\scrD=\scrF_H\left\{R(w): w\in \scrB\right\}\subset\Sigma(\scrW,1)$. 
Then for any $X>0$, 
\begin{equation}\label{convi}
\lim_{L\to\infty}
\lambda\left(\left\{ x\in\scrX : L^{-n/k} |\xi_1(x,L)|> X ,  \; \omega_1(x,L)\in\scrD \right\}\right)  = \int_{(\B_X^k)^c} \int_{\scrD} \psi(\xi,\omega)\,  d\nu(\omega)\, d\xi,
\end{equation}
with the $\L^1$ probability density
\begin{equation}\label{densi}
\psi(\xi,\omega) = \one(\xi\in\scrR(\omega,1)) .
\end{equation}
\end{thm}

\begin{proof}
By the scaling property we have for $t=(\log L)/k$,
$$
\left\{x\in \scrX: L^{-n/k}|\xi_1(x,L)|>X,\, \omega_1(x,L)\in \scrD\right\}=\varphi_{t}(\scrE_{X,\scrD}),
$$
where
\begin{align*}
\scrE_{X,\scrD}:=\left\{x\in \scrX: |\xi_1(x,1)|>X,\, \omega_1(x,1)\in \scrD\right\}.
\end{align*}
We thus have 
\begin{align*}
\lim_{L\to\infty}
\lambda\left(\left\{ x\in\scrX : L^{-n/k} |\xi_1(x,L)|> X ,  \; \omega_1(x,L)\in\scrD \right\}\right)  
=\mu(\scrE_{X,\scrD}),
\end{align*}
provided that $\mu(\partial\scrE_{X,\scrD})=0$. Then from the above expression applying \eqref{dissi} with $f=\chi_{\scrE_{\scrD}}$,
we obtain the formula in \eqref{convi}, \eqref{densi}.
Below we prove $\mu(\partial\scrE_{X,\scrD})=0$. Let 
$$
\scrE_X:=\left\{x\in \scrX: |\xi_1(x,1)|>X\right\},\quad \text{and}\quad \scrE_{\scrD}:=\left\{x\in \scrX: \omega_1(x, 1)\in \scrD \right\}.
$$
Since $\scrE_{X,\scrD}=\scrE_X\cap \scrE_{\scrD}$, we have $\partial\scrE_{X,\scrD}\subset \partial\scrE_X\cup\partial\scrE_{\scrD}$. It thus suffices to show $\mu(\partial\scrE_X)=\mu(\partial\scrE_{\scrD})=0$. 
For the first statement note that \eqref{U2} implies 
\begin{align*}
\scrE_X=\left\{x\in \scrX: \#\left\{j: \xi_j(x,1)\in \B_X^k\right\}=0\right\}=\left\{\Gamma g\in \scrX: \hat{\ZZ}^ng\cap \scrC(\B_X^k, \scrW)=\emptyset\right\}.
\end{align*}
Hence by the same arguments as in the proof of Theorem \ref{thm2var}, we have $\mu(\partial\scrE_X)=0$.

Next, we show $\mu(\partial\scrE_{\scrD})=0$. 
For any $v=(v',v'', v_n)\in \RR^k\times \RR^m\times \RR_{>0}$, we write $s_v:=\frac{v'}{v_n}$. Moreover, for any Borel $\scrG\subset \RR^m\times \RR_{>0}$ we use the notation $\scrF_{H}\scrG$ to denote the set $\scrF_H\{R(w): w\in \scrG\}$. 
We also introduce the following set:
$$
\scrN=\left\{\Gamma g\in \scrX: \hat{\ZZ}^ng\cap (\RR^k\times (\partial\scrB\cup \partial\scrW))\neq \emptyset\right\}.
$$
Since $\vol_{\RR^{m+1}}(\partial\scrB\cup \partial\scrW)=0$, we have $\mu(\scrN)=0$ by Siegel's volume formula \cite{Siegel1945}. 
Now take $x=\Gamma g\in \partial\scrE_{\scrD}$. Since $\xi_1(x,1)$ exists for $\mu$-a.e. $x$ and $\mu(\scrN)=0$, we may assume $\xi_1(x,1)$ exists and $x\not\in \scrN$. In particular, this implies $\omega_1(x,1)\notin \scrF_H((\partial \scrB\cup\partial\scrW)\cap \scrW)$, which means that either $\omega_1(x,1)\in \scrF_H\scrB^{\circ}$ or $\omega_1(x,1)\in \scrF_H(\scrW^{\circ}\cap (\scrB^c)^{\circ})$, where $\scrB^c:=(\RR^m\times \RR_{>0})\setminus \scrB$.  

We now show that $x$ is contained in some set of $\mu$-measure zero.  Since $\Gamma g\in\partial\scrE_{\scrD}$, there exist $\{g_j\}, \{g'_j\}\subset G$  such that $\lim_{j\to\infty}g_j=\lim_{j\to\infty}g'_j=g$, $x_j:=\Gamma g_j\in \scrE_{\scrD}$ and $x'_j:=\Gamma g_j'\notin \scrE_{\scrD}$ for all $j\geq 1$. 
If $\omega_1(x,1)\in \scrF_H\scrB^{\circ}$, then there exists $w\in \hat{\ZZ}^n$ such that  $wg\in \RR^k\times \scrB^{\circ}$ and $s_{wg}=\xi_1(x,1)$. Since $\scrB^{\circ}$ is open, up to removing finitely many elements, we may assume $wg_j, wg_j'\in \RR^k\times  \scrB^{\circ}$ for all $j\geq 1$. Since $x_j'\not\in \scrE_{\scrD}$, there exists $w_j\in \hat{\ZZ}^n$ such that $w_jg'_j\in \RR^k\times (\scrW\setminus \scrB)$ and $|s_{w_jg_j'}|\leq |s_{wg_j'}|$. 
Note also that $s_{wg_j'}\to s_{wg}$, thus 
$\sup_j |s_{w_jg_j'}|<\infty$.
Moreover, since $w_jg_j'\in \RR^k\times \scrW$ and $\scrW$ is bounded, we have $\sup_j\|w_jg_j'\|<\infty$, 
which then implies that $\sup_j\|w_j\|<\infty$ (since $g_j'\to g$). Thus there are only finitely many choices of such $w_j$'s. Hence up to passing to a subsequence we may assume $w_j=w'$ for all $j\geq 1$ and for some $w'\in \hat{\ZZ}^n$. Since $w'g_j'\in \RR^k\times (\scrW\setminus\scrB)$, taking $j\to\infty$ and using the assumption that $x\not\in \scrN$, we have $w'g\in \RR^k\times (\scrW\setminus\scrB^{\circ})$. This, together with the facts that $wg\in \RR^k\times \scrB^{\circ}$ and $\scrW\subset \RR^m\times \RR_{>0}$, implies that $w'\neq \pm w$. Moreover, since $|s_{w'g_j'}|\leq |s_{wg_j'}|$, taking $j\to\infty$ we get $|s_{w'g}|\leq |s_{wg}|$. Since $s_{wg}=\xi_1(x,1)$ and $w'g\in \RR^k\times \scrW$, we must have $|s_{w'g}|=|s_{wg}|$. This shows that $x$ is contained in the set 
$$
\scrM:=\left\{y=\Gamma g': \text{$\exists\, (v,w)\in \hat{\ZZ}^ng'\times \hat{\ZZ}^ng'$ lin. ind. s.t. $(v,w)\in \scrP$}\right\},
$$
with
$$
\scrP:=\left\{(v,w)\in (\RR^k\times \scrW)\times (\RR^k\times \scrW): |v'|w_n=|w'|v_n\right\}.
$$
Since $\vol_{\RR^{2n}}(\scrP)=0$, by Rogers' second moment formula on the Siegel transform \cite{Rogers1955} we have $\mu(\scrM)=0$.

Next, if $\omega_1(x,1)\in \scrF_H(\scrW^{\circ}\cap (\scrB^c)^{\circ})$, then there exists some $w\in \hat{\ZZ}^ng$ such that $wg\in \RR^k\times ( \scrW^{\circ}\cap (\scrB^c)^{\circ})$ and $s_{wg}=\xi_1(x,1)$. 
Since $\scrW^{\circ}\cap (\scrB^c)^{\circ}$ is open, up to removing finitely many elements we have $wg_j, wg_j'\in \RR^k\times (\scrW^{\circ}\cap (\scrB^c)^{\circ})$ for all $j\geq 1$. 
Since $x_j\in \scrE_{\scrD}$, there exists $w_j\in \hat{\ZZ}^n$ such that $w_jg_j\in \RR^k\times \scrB$ and $|s_{w_jg_j}|\leq |s_{wg_j}|$. Then by similar arguments as above we see that there exists some $w'\in \hat{\ZZ}^n$ such that $w_j=w'$ for infinitely many $j$. Taking $j\to\infty$ and using the assumption that $x\not\in \scrN$ we get $w'g\in \RR^k\times \scrB^{\circ}$ and $|s_{w'g}|\leq |s_{wg}|$. This then (using also that $wg\in \RR^k\times ( \scrW^{\circ}\cap (\scrB^c)^{\circ})$) shows that $w'\neq \pm w$.  
Now since $s_{wg}=\xi_1(x,1)$ and $w'g\in \RR^k\times \scrB^{\circ}\subset \RR^k\times \scrW$, we have $|s_{w'g}|=|s_{wg}|$, showing that $x\in \scrM$.  This finishes the proof that $\mu(\partial \scrE_{\scrD})=0$. Hence we have $\mu(\partial\scrE_{X,\scrD})=0$ and the theorem follows.
\end{proof}
From relation \eqref{convi} we see that the limit density in Corollary \ref{cor1} is related to $\psi(\xi,\omega)$ by
\begin{equation}\label{psi0intermsofpsi}
\psi_0(r) = k\,  \vol_{\RR^k}(\B_1^k)\,r^{k-1} 
\int_{\S_1^{k-1}} \int_{\Sigma(\scrW,1)} \psi(r \theta,\omega) \, d\nu(\omega) \, d\Omega(\theta).
\end{equation}
Here $\Omega$ is the uniform probability measure on the unit sphere $\S_1^{k-1}=\left\{s\in\RR^k: |s|=1\right\}$ in $\RR^k$,
which in the case of a general norm $|\cdot|$ may be defined as
the pushward measure of $\vol_{\RR^k}(\B_1^k)^{-1} \vol_{\RR^k}|_{\B_1^k\setminus\{0\}}$ 
under the map $\B_1^k\setminus\{0\}\to S_1^{k-1}$ sending $v$ to $v/|v|$.

\section{Entry times to a cuspidal neighbourhood}
\label{entrytimessec}

For any Borel subset $\mathcal{C}\subset \R^n$ set 
\begin{align}\label{equ:hitset}
\scrH(\mathcal{C}):=\left\{\Gamma g\in \scrX: {\hat{\ZZ}^ng}\cap  \mathcal{C} \neq \emptyset\right\}.
\end{align}
In this section we will study the hitting time statistics of $h_s(x)$ to the family of regions $\left\{\scrH(L^{-1}\mathcal{C})\right\}_{L>0}$ when $\mathcal{C}$ is of positive Lebesgue measure. Note that $\scrH(\scrW,L)=\scrH(L^{-1}(\{0\}\times \scrW))$. These regions are of positive $\mu$-measure, and the set 
\begin{align*}
\left\{s\in \R^k: h_s(x)\in \scrH(L^{-1}\mathcal{C})\right\}
\end{align*}
may not be discrete for generic $x\in \scrX$. Thus instead of studying the hitting time parameters directly as before, we consider the norm of these parameters: For any subset $\scrE\subset \scrX$, let
\begin{align}\label{def:hittime}
r_1(x,\scrE):=\inf\,\left\{|s|: h_s(x)\in \scrE\right\},
\end{align}
with the convention that $r_1(x,\scrE)=\infty$ if $h_s(x)\notin \scrE$ for all $s\in \RR^k$. 

Since any lattice is symmetric under reflection in the origin, we have $\scrH(\mathcal{C})=\scrH(\mathcal{C}\cup(-\mathcal{C}))$. Thus 
without loss generality, we may assume throughout this section that $\mathcal{C}$ is symmetric under reflection in the origin, i.e. $\mathcal{C}=-\mathcal{C}$.
We will also assume $\mathcal{C}\subset \RR^n$ is compact with non-empty interior and boundary of Lebesgue measure zero. On the other hand, we have that
\begin{align}\label{equ:hitrela}
r_1(x, \scrH(L^{-1}\mathcal{C}))=r_1(x, \scrH(L^{-1}\mathcal{C}^+))\quad \forall\,  L>0,\,  x\in \scrH(L^{-1}\scrC\cap \{y\in \RR^n: y_n=0\})^c,
\end{align}
where
\begin{align}\label{equ:Cpositive}
\mathcal{C}^+:=\mathcal{C}\cap \{y\in \R^n: y_n>0\}.
\end{align}
In order to verify \eqref{equ:hitrela}, it suffices to show for any $x=\Gamma g\in \scrH(L^{-1}\scrC\cap \{y\in \RR^n: y_n=0\})^c$ and $s\in \RR^k$, $h_s(x)\in \scrH(L^{-1}\scrC)$ if and only if $h_s(x)\in \scrH(L^{-1}\scrC^+)$. Clearly $h_s(x)\in \scrH(L^{-1}\scrC^+)$ implies $h_s(x)\in \scrH(L^{-1}\scrC)$. If $h_s(x)\in \scrH(L^{-1}\scrC)$, then there exists some $v\in \hat{\ZZ}^ng$ such that $vU(s)\in L^{-1}\scrC$. Now we must have $v_n\neq 0$, since otherwise $v = vU(s)\in \hat{\ZZ}^n g\cap L^{-1}\scrC\cap\{y_n=0\}$, contradicting our assumption on $x$. Hence, using also $\scrC=-\scrC$ , we
have either $vU(s)\in L^{-1}\scrC^+$  or $-vU(s)\in L^{-1}\scrC^+$ and so $h_s(x)\in \scrH (L^{-1}\scrC^+)$.


In view of \eqref{equ:hitrela} it is natural to also consider the family of sections $\left\{\scrH(L^{-1}\mathcal{C}^+)\right\}_{L>0}$.
Note that these sets can be parameterised as follows:
\begin{equation}\label{sectBBapp}
\scrH(L^{-1}\mathcal{C}^+) = \Gamma\backslash\Gamma H \big\{ M(y):  y\in  L^{-1}\mathcal{C}^+\big\} ,
\end{equation}
where
\begin{equation}\label{Myapp}
M(y) = \begin{pmatrix} y_n^{-(m+1)/k} 1_k & 0 & 0 \\ 0  &  y_n 1_m & 0 \\  y' & y''  & y_n \end{pmatrix}
\end{equation}
with $y=(y',y'',y_n)$. This choice of $M(y)$ satisfies $(0,1)M(y)=y$, so that the map in \eqref{My0} gives a parametrization of $G$, and we have $M(0,w)=R(w)$ and, more generally,
\begin{equation}\label{thenapp}
M(y',y'',y_n)=R(y'',y_n) U(-y_n^{-1} y'). 
\end{equation}

The main result of this section is the following analogue of Corollary \ref{cor1} regarding the  hitting times to the family of sections $\left\{\scrH(L^{-1}\mathcal{C}))\right\}_{L>0}$. 
\begin{cor}\label{cor3app}
Let $\lambda$ be an admissible probability measure on $\scrX$. 
 Let $\mathcal{C}\subset \RR^n$ be a compact set with non-empty interior and boundary of Lebesgue measure zero,
satisfying $\scrC=-\scrC$. 
Then for any $X>0$, 
\begin{equation}\label{equ:cuspapp}
\lim_{L\to\infty}
\lambda\left(\left\{ x\in\scrX : L^{-n/k} r_1(x,\scrH(L^{-1}\mathcal{C}))>X \right\}\right)  =  \int_X^\infty \psi_0(r)\, dr,
\end{equation}
with the probability density $\psi_0\in\L^1(\RR_{>0})$  as in Corollary \ref{cor1} and  with $\scrW=\mathrm{pr}(\mathcal{C}^+)$, where $\mathrm{pr}: \RR^n\to \RR^{m+1}$ is as in Theorem \ref{mainthm} and $\scrC^+$ is as in \eqref{equ:Cpositive}.
\end{cor}
The proof of \eqref{equ:cuspapp} naturally divides into two parts, namely, the upper and lower bounds. We will first show that we may replace $r_1(x,\scrH(L^{-1}\mathcal{C}))$ by $r_1(x,\scrH(L^{-1}\mathcal{C}^+))$ in the right hand side of \eqref{equ:cuspapp}. Then the general strategy is to use \eqref{thenapp} to relate $r_1(x,\scrH(L^{-1}\mathcal{C}^+))$ to the entry times for the sections $\scrH(\scrW,L)$ with $\scrW={{\mathrm{pr}(\mathcal{C}^{+})}}$ as above. For the upper bound, we follow the strategy in \cite[Cor.~2.1]{MP25} using results on the impact statistics developed in last section. For the lower bound, due to some difficulties occurring in the current higher rank situation, the approach in \cite{MP25} does not apply. Instead we use a more direct approach applying the conjugation relation in \eqref{commute} directly to the hitting times for $\scrH(L^{-1}\mathcal{C}^+)$

\subsection{Preliminaries on entry times}
We first show that if $r_1(x,\scrH(L^{-1}\mathcal{C}^+))<\infty$, then there must exist some $s\in \RR^k$ such that $h_s(x)\in \scrH(L^{-1}\mathcal{C}^+)$ and $|s|=r_1(x,\scrH(L^{-1}\mathcal{C}^+))$.
\begin{lem}\label{lem:hittingapp}
Assume $\mathcal{B}\subset \R^n$ is compact and let $\mathcal{B}^+=\mathcal{B}\cap \{y\in \R^n: y_n>0\}$. 
For any $x\in \scrX$ such that $r_1(x,\scrH(\mathcal{B}^+))<\infty$, there exists some $s\in \R^k$ satisfying $h_s(x)\in \scrH(\mathcal{B}^+)$ and $|s|=r_1(x,\scrH(\mathcal{B}^+))$.  
 \end{lem}
 \begin{proof}
Assume $x=\Gamma g$. By definition, there exist $\{s_l\}\subset \RR^k$ satisfying $h_{s_l}(x)\in \scrH(\mathcal{B}^+)$ and $|s_l|\searrow r_1(x, \scrH(\mathcal{B}^+))$. By passing to a subsequence, we may assume there exists some $s\in \RR^k$ such that $s_l\to s$; hence $|s|=r_1(x, \scrH(\mathcal{B}^+))$. We will show $h_s(x)\in \scrH(\mathcal{B}^+)$. Since $h_{s_l}(x)\in \scrH(\mathcal{B}^+)$, there exists some  vector $v_l\in {\hat{\ZZ}^ng}$ such that $v_lU(s_l)\in \mathcal{B}^+$. By compactness of $\mathcal{B}$, by passing to a subsequence we may assume there exists some $w\in \mathcal{B}\cap \{y\in \R^n: y_n\geq 0\}$ such that $v_lU(s_l)\to w$. 
Hence $\lim_{l\to\infty}v_l=\lim_{l\to\infty} wU(-s_l)=wU(-s)$. By discreteness of ${\hat{\ZZ}^ng}$, $wU(-s)=v_l$ for all $l$ sufficiently large, i.e. $v:=wU(-s)\in {\hat{\ZZ}^ng}$. Hence 
$$
w=vU(s)\in \mathcal{B}\cap \{y\in \R^n: y_n\geq 0\}\cap {\hat{\ZZ}^ng}U(s).
$$
This shows that $h_s(x)\in \scrH(\mathcal{B}^+)$ unless $w_n=0$. But if $w_n=0$ then $v=wU(-s)=w$; this would imply that $v_l=w$ for all sufficiently large $l$, and we would then get a contradiction against the fact that $v_{l,n}>0$, which holds since $v_lU(s_l)\in \scrB^+$ and $v_l$ and $v_lU(s_l)$ have the same $n$-th coordinate. Hence $w_n=0$ cannot hold, and the lemma is proved. 
 \end{proof}

We now start to deduce relations between entry times to $\scrH(L^{-1}\mathcal{C}^+)$ and $\scrH(\scrW,L)$ with  $\scrW=\mathrm{pr}(\mathcal{C}^{+})$. Note that $w\in \scrW$ if and only if there exists some ${u_w}\in \RR^k$ such that $(u_w,w)\in \mathcal{C}^{+}$. 
This, together with \eqref{thenapp}, implies that for any $x\in \scrX$ and $L>0$, $\xi_1(x,L)$ exists if and only if $r_1(x, \scrH(L^{-1}\mathcal{C}^+))<\infty$. 
Moreover, when they both exist (which holds for almost every $x\in \scrX$) we may write
\begin{align}\label{equ:1sthitsapp}
 h_{\xi_1(x,L)}(x)=\Gamma h R(w),
\end{align}
for some $h\in \scrF_H$ and $w\in  L^{-1}\scrW$, i.e. $(u_w,w)\in  L^{-1}\mathcal{C}^+$ for some $u_w\in \RR^k$.  Set $\delta:=-w_{m+1}^{-1}u_{w}$ 
Then by \eqref{thenapp} we have $h_{\xi_1(x,L)+\delta}(x)\in \scrH(L^{-1}\mathcal{C}^+)$. This implies that 
\begin{align}\label{bo3app}
r_1(x, \scrH(L^{-1}\mathcal{C}^+))\leq |\xi_1(x,L)+\delta|\leq |\xi_1(x,L)|+w_{m+1}^{-1}|u_w|.
\end{align}

Using the parameterisation \eqref{paraL} for $\scrH(\scrW,L)$, we may parameterise the impact parameters as follows:
\begin{align*}
\omega_j(x,L)= h_j R(w^j(x,L))\in\Sigma(\scrW,\,1)
\end{align*}
where $ h_j\in \scrF_H$ and $w^j(x,L)=(w^j_1(x,L),\cdots, w^j_{m+1}(x,L))\in \RR^m\times \RR_{>0}$ satisfies
\begin{align*}
h_{\xi_j(x,L)}(x)=\Gamma h_jR(L^{-1}w^j(x,L))\in \scrH(\scrW,L). 
\end{align*}
Keep the notation in \eqref{equ:1sthitsapp}. 
Let 
$$
C=C(\mathcal{C}^+):=\sup\left\{|u|: \exists\, w\in \RR^{m+1}\,\text{such that}\, (u,w)\in \mathcal{C}^+\right\}.
$$
Note that $0<C<\infty$ since $\mathcal{C}=-\scrC$ is compact and has non-empty interior.
Then  \eqref{bo3app} and the relation $w_{m+1}=L^{-1}w^1_{m+1}(x,L)$ imply that 
\begin{align}\label{bo3capp}
r_1(x, \scrH(L^{-1}\mathcal{C}^+))- |\xi_1(x,L)|\leq w_{m+1}^{-1}|u_w|\leq  C\big/w^1_{m+1}(x,L).
\end{align}
We will use this relation to prove the upper bound in \eqref{equ:cuspapp}. 

On the other hand, by Lemma \ref{lem:hittingapp} there exists some $s_0\in \RR^k$ satisfying $h_{s_0}(x)\in \scrH(L^{-1}\mathcal{C}^+)$ and $|s_0|=r_1(x,\scrH(L^{-1}\mathcal{C}^+))$. Write $h_{s_0}(x)=\Gamma hM(y)$ for some $h\in \scrF_H$ and $y=(y',y'',y_n)\in L^{-1}\mathcal{C}^+$. Then similarly, we have by \eqref{thenapp} that $h_{s_0+\delta}(x)\in \scrH(\scrW,L)$ with $\delta:=y_n^{-1}y'$. However, unlike the rank one case considered in \cite{MP25}, 
this hit does not necessarily correspond to the first hit of $h_s(x)$ 
to $\scrH(\scrW,L)$.  For this reason, we adapt a more direct approach to prove the lower bound in \eqref{equ:cuspapp}. We apply the conjugation relation in \eqref{commute} for $\scrH(L^{-1}\mathcal{C}^+)$ and consider entry times to the sets $\varphi_{-t}\scrH(L^{-1}\mathcal{C}^+)$. 
Note that similar to \eqref{secscale2}, for any Borel subset $\scrE\subset \scrX$, $x\in \scrX$ and $t\in \RR$, \eqref{commute} implies 
\begin{align}\label{equ:bo4capp}
r_1(x,\scrE)= e^{nt}r_1(\varphi_{-t}(x),\varphi_{-t}\scrE).
\end{align}
Moreover,  for any $x=\Gamma g\in \scrX$ and $t\in \R$,  
$$
x\in \varphi_{-t}\scrH(e^{-kt}\mathcal{C}^+)\quad \Leftrightarrow \quad \exists\, v=(v',v'',v_n)\in {\hat{\ZZ}^ng}\,\ :\, \ 
(e^{nt}v',v'',v_n)\in \mathcal{C}^{+}.
$$
That is, 
\begin{align}\label{equ:conrelaapp}
\varphi_{-t}\scrH(e^{-kt}\mathcal{C}^+)=\scrH(\scrA_t),\quad \text{with $\scrA_t:=\mathcal{C}^{+}\left(\begin{smallmatrix}
e^{-nt}I_k & &\\
 & I_m & \\
 & & 1\end{smallmatrix}\right)$}.
\end{align}
We also consider entry times to another family of sets $\scrH(\widetilde{\scrA}_t)$, where
$$
\widetilde{\scrA}_t:=\widetilde{\mathcal{C}}^{+}\left(\begin{smallmatrix}
e^{-nt}I_k & &\\
 & I_m & \\
 & & 1\end{smallmatrix}\right),\quad \text{with $\widetilde{\mathcal{C}}^{+}:=\left\{(\delta v', v'', v_n): (v', v'', v_n)\in \mathcal{C}^{+},\, 0\leq \delta\leq 1\right\}$ }.
$$
Note that $x=\Gamma g\in \scrH(\widetilde{\scrA}_t)$ if and only if there exists some  $v=(v',v'',v_n)\in {\hat{\ZZ}^ng}$ 
such that 
$
(e^{nt} v', v'', v_n)\in \widetilde{\mathcal{C}}^{+}.
$
Moreover, note that 
$$
\mathrm{pr}(\widetilde{\mathcal{C}}^{+})=\mathrm{pr}({\mathcal{C}}^{+})=\scrW,\quad \scrH(\scrW,1)=\scrH(0\times \scrW),
$$
and for any $w=(w'', w_n)\in \scrW$ and $u\in \RR^k$, 
\begin{align*}
(0,w'', w_n)U(-w_n^{-1}u)=(u, w'', w_n). 
\end{align*}
This shows that if $\xi_1(x,1)$ exists, then $r_1(x, \scrH(\scrA_t))$ and $r_1(x,\scrH(\widetilde{\scrA}_t))$ are finite for all $t\in \RR$. And similarly, if $r_1(x, \scrH(\scrA_t))$ or $r_1(x,\scrH(\widetilde{\scrA}_t))$ is finite for some $t\in \RR$, then $\xi_1(x,1)$ also exists. 
Next, note that $\scrH(\scrA_t)\subset \scrH(\widetilde{\scrA}_t)$ and $\scrH(\scrW,1)\subset\scrH(\widetilde{\scrA}_t)$. 
This implies for any $x\in \scrX$ such that $\xi_1(x,1)$ exists, 
\begin{align}\label{equ:apbd1app}
r_1(x,\scrH(\widetilde{\scrA}_t))\leq \min\bigl\{ r_1(x, \scrH(\scrA_t)), |\xi_1(x,1)|\bigr\}.
\end{align}
Finally, note that $\scrH(\widetilde{\scrA}_t)$ is decreasing in $t$, thus $r_1(x,\scrH(\widetilde{\scrA}_t))$ is increasing in $t$. 
We prove the following limiting relation between $r_1(x,\scrH(\widetilde{\scrA}_t))$ and $\xi_1(x,1)$. 
\begin{lem}\label{lem:relhitapp}
For any $x\in \scrX$ such that $\xi_1(x,1)$ exists, 
\begin{align}\label{equ:bo5capp}
\lim_{t\to\infty}r_1(x,\scrH(\widetilde{\scrA}_t))
=|\xi_1(x,1)|. 
\end{align}
\end{lem}
\begin{remark}
Using similar arguments we can also show 
$
\lim_{t\to\infty}r_1(x, \scrH(\scrA_t))=|\xi_1(x,1)|.
$
However, we will not need this fact for our later proofs. 
\end{remark}
\begin{proof}[Proof of Lemma \ref{lem:relhitapp}] 
Take $x= \Gamma g\in \scrX$ such that $\xi_1(x,1)$ exists. 
Then again by the above discussion, $r_1(x,\scrH(\widetilde{\scrA}_t))$ exist for all $t\in \RR$. Note that 
$\widetilde{\scrA}_t=\mathcal{B}_t\cap \{y\in \RR^n: y_n>0\}$ with the compact set 
\begin{align}\label{equ:cpttilde}
\mathcal{B}_t:=\left\{(e^{-nt}\delta v',v'', v_n): (v',v'',v_n)\in\mathcal{C},\, 0\leq \delta\leq 1\right\}. 
\end{align}
Thus Lemma \ref{lem:hittingapp} applies to $\scrH(\widetilde{\scrA}_t)$, showing that there exists some $s_t\in \RR^k$ satisfying $|s_t|=r_1(x,\scrH(\widetilde{\scrA}_t))$ and $h_{s_t}(x)\in \scrH(\widetilde{\scrA}_t)$.
Since $|s_t|\leq |\xi_1(x,1)|$ (cf. \eqref{equ:apbd1app}) and $|s_t|$ is increasing in $t$, it suffices to show $\lim_{t\to\infty}|s_t|\geq |\xi_1(x,1)|$. By definition we have that 
for any $t>0$, there exists some  $v_t=(v_t', v_t'', v_{n,t})\in {\hat{\ZZ}^ng}$ such that 
\begin{align}\label{equ:phitcond2app}
(e^{nt}(v_t'-s_tv_{n,t}), v_t'', v_{n,t})\in \widetilde{\mathcal{C}}^{+}.
\end{align}
Since $\{|s_t|\}$ is uniformly bounded from above by $|\xi_1(x,1)|$ and $\widetilde{\mathcal{C}}^{+}$ is bounded, \eqref{equ:phitcond2app} implies that $\{v_t\}\subset {\hat{\ZZ}^ng}$ is uniformly bounded; in particular, it is a finite set. Thus there exists some $\delta_x>0$ such that 
$v_{n,t}>\delta_x$ for all $t>0$. Write $\ell_t:=v_t'-s_tv_{n,t}$. 
Note that \eqref{equ:phitcond2app} also implies that $(v_t'', v_{n,t})\in \scrW$ (recall $\mathrm{pr}(\widetilde{\mathcal{C}}^{+})=\scrW$) and {{if $\ell_t\neq 0$ then there exists some $0< \delta\leq 1$ such that $(\delta^{-1}e^{nt}\ell_t, v_t'', v_{n,t})\in \mathcal{C}^+$. Thus 
$
\delta^{-1}e^{nt}|\ell_t|\leq C.
$ 
This shows that 
$$
|\ell_t|\leq C \delta e^{-nt}\leq Ce^{-nt}
$$ 
if $\ell_t\neq 0$. This estimate clearly also holds if $\ell_t=0$. }}
We thus get $h_{v_{n,t}^{-1}\ell_t+s_t}(x)\in \scrH(\scrW,1)$, and
$$
|\xi_1(x,1)|\leq |s_t|+v_{n,t}^{-1}|\ell_t|\leq |s_t|+C\delta_x^{-1}e^{-nt}.
$$
This shows that $|\xi_1(x,1)|\leq \lim_{t\to\infty}|s_t|$ as desired. 
\end{proof}


\subsection{Proof of Corollary \ref*{cor3app}}
In this section we collect results from the previous section to prove Corollary \ref{cor3app}. 
\begin{proof}
Let $\scrC_0:=\scrC\cap \{y\in \RR^n: y_n=0\}$.  
 By \eqref{equ:hitrela} we have for any $L>0$,
\begin{align*}
0\leq \lambda(\{ x\in\scrX : L^{-n/k} r_1(x,\scrH(L^{-1}\mathcal{C}^+))>X \} )- \lambda(\{ x\in\scrX : L^{-n/k} r_1(x,\scrH(L^{-1}\mathcal{C}))>X\})  \leq \lambda(\scrH(L^{-1}\scrC_0)). 
\end{align*}
Now fix $R>0$ so large such that the box $\scrB_0:=[-R,R]^{n-1}\times \{0\}$ contains $\scrC_0$. Note that $\scrB_0\subset \scrB_0\diag(\lambda_1,\cdots, \lambda_n)$ for any $\lambda_1,\cdots, \lambda_n\geq 1$. In particular, for any $L\geq 1$, taking $t\geq 0$ such that $e^{kt}=L$ we have $L^{-1}\scrB_0\subset \scrB_0\Phi(t)$. Thus 
\begin{align*}
\limsup_{L\to\infty}\lambda(\scrH(L^{-1}\scrC_0))&\leq \limsup_{L\to\infty}\lambda(\scrH(L^{-1}\scrB_0))\leq \limsup_{t\to\infty}\lambda(\scrH(\scrB_0\Phi(t)))\\
&=\limsup_{t\to\infty}\lambda(\varphi_t\scrH(\scrB_0))\leq \mu(\scrH(\scrB_0))=0.
\end{align*}
Here for the last inequality we applied the fact that $\scrH(\scrB_0)\subset \scrX$ is closed (since $\scrB_0$ is compact) and the implication ``(1)$\Rightarrow$ (2)" in Remark \ref{rmk:port}. 
Therefore, it suffices to show
\begin{align}\label{equ:cuspapp+}
\lim_{L\to\infty}\lambda\left(\left\{ x\in\scrX : L^{-n/k} r_1(x,\scrH(L^{-1}\mathcal{C}^+))>X \right\}\right)  =  \int_X^\infty \psi_0(r)\, dr.
\end{align}
By \eqref{bo3capp} we have for any $0<\delta\leq 1$,
\begin{align*}
&\left\{ x : L^{-n/k} r_1(x,\scrH(L^{-1}\mathcal{C}^+)) > X, \, w_{m+1}^1(x,L)>\delta \right\}\\
&\subset \left\{ x : L^{-n/k} |\xi_1(x,L)| > X-CL^{-n/k}\delta^{-1},\, w_{m+1}^1(x,L)>\delta \right\}.
\end{align*}
Hence  $\left\{ x : L^{-n/k} r_1(x,\scrH(L^{-1}\mathcal{C}^+)) > X \right\}$ is a subset of
\begin{align*}
 \left\{ x : L^{-n/k} |\xi_1(x,L)| > X-CL^{-n/k}\delta^{-1},\, w_{m+1}^1(x,L)>\delta \right\}\cup \left\{x : w_{m+1}^1(x,L)\leq \delta\right\}.
\end{align*}
Therefore, by Theorem \ref{cor1B} we have for any $0<\delta\leq 1$,
\begin{align*}
&\limsup_{L\to\infty}\lambda\left(\left\{ x\in\scrX : L^{-n/k} r_1(x,\scrH(L^{-1}\mathcal{C}^+))>X \right\}\right)\\
 &\leq \int_{(\B_X^k)^\mathrm{c}} \int_{\scrD_\delta} \psi(\xi,\omega)  d\nu(\omega) d\xi+\int_{\RR^k} \int_{(\scrD_\delta)^c} \psi(\xi,\omega)  d\nu(\omega) d\xi,
\end{align*}
where
\begin{align*}
\scrD_{\delta} := \scrF_H \big\{ R(w) :  w \in \scrW,\, w_{m+1}>\delta  \big\} \subset \Sigma( \scrW,1),
\end{align*}
and where we used the fact that the integral $\int_{(\B_X^k)^\mathrm{c}} \int_{\scrD_\delta} \psi(\xi,\omega)  d\nu(\omega) d\xi$ depends continuously on $X$.
Since $\psi$ is an $\L^1$ probability density function, we have by the dominated convergence theorem that
\begin{align*}
\lim_{\delta\to 0^+}\int_{\RR^k} \int_{(\scrD_\delta)^c} \psi(\xi,\omega)  d\nu(\omega) d\xi=\int_{\RR^k} \int_{\Sigma(\scrW, 1)} \psi(\xi,\omega)\lim_{\delta\to 0^+}\chi_{\RR^k\times \scrD_{\delta}^c}(\xi, \omega)  d\nu(\omega) d\xi=0,
\end{align*}
where for the last equality we used that 
$
\lim_{\delta\to 0^+} \chi_{\RR^k\times \scrD_{\delta}^c}=\chi_{\RR^k\times \emptyset}=0.
$
Thus we get
\begin{align*}
&\limsup_{L\to\infty}\lambda\left(\left\{ x\in\scrX : L^{-n/k} r_1(x,\scrH(L^{-1}\mathcal{C}^+))>X \right\}\right)\\
 &\leq \lim_{\delta\to 0^+}\int_{(\B_X^k)^\mathrm{c}} \int_{\scrD_\delta} \psi(\xi,\omega)  d\nu(\omega) d\xi+\lim_{\delta\to 0^+}\int_{\RR^k} \int_{(\scrD_\delta)^c} \psi(\xi,\omega)  d\nu(\omega) d\xi\\
 &\leq \Psi_0(\B_X^k),
\end{align*}
proving the upper bound in \eqref{equ:cuspapp+}. 

For the lower bound, take $t>0$ such that $L=e^{kt}$. Then 
by \eqref{equ:bo4capp} and the relation \eqref{equ:conrelaapp},
\begin{align*}
\left\{ x\in\scrX : L^{-n/k} r_1(x,\scrH(L^{-1}\mathcal{C}^+)) > X \right\}
=\left\{ x\in\scrX :  r_1(\varphi_{-t}(x), \scrH(\scrA_t)) > X \right\}=\varphi_t(\scrE_t),
\end{align*}
where 
\begin{align*}
\scrE_t:=\left\{x\in \scrX:    r_1(x, \scrH(\scrA_t)) > X  \right\}.
\end{align*}
Set also
\begin{align*}
\scrE'_t:=\left\{x\in \scrX:   r_1(x,\scrH(\widetilde{\scrA}_t)) > X  \right\}
\end{align*}
and
\begin{align*}
\scrE:=\left\{x\in \scrX:   |\xi_1(x,1)| > X\  \text{or $\xi_1(x,1)$ does not exist}\right\}.
\end{align*}
Then by \eqref{equ:apbd1app} and \eqref{equ:bo5capp}  we have $\scrE'_t\subset \scrE_t$ and $\scrE'_t\subset \scrE$ and $\bigcup_t\scrE_t'=\scrE$. 
We also note that $\scrE'_t$ is increasing in $t$. 
Thus 
\begin{align}\label{equ:monconapp}
\lim_{t\to\infty}\mu(\scrE'_t)=\mu(\scrE). 
\end{align}
Moreover, for any fixed $t_0>0$ we have 
\begin{align*}
\liminf_{t\to\infty}\lambda\left(\left\{ x\in\scrX : L^{-n/k} r_1(x,\scrH(L^{-1}\mathcal{C}^+)) > X \right\}\right)&=\liminf_{t\to\infty}\lambda(\varphi_t(\scrE_t))\geq \liminf_{t\to\infty}\lambda(\varphi_t(\scrE'_t))\\
&\geq \liminf_{t\to\infty}\lambda(\varphi_t(\scrE'_{t_0})).
\end{align*}
Applying Lemma \ref{lem:hittingapp} to the compact set $\scrB_t$ given in \eqref{equ:cpttilde},
we have for any $t>0$, $x\in \scrE_t'$ if and only if $h_s(x)\in \scrH(\widetilde{\scrA}_t)^c$ for all $|s|\leq X$. That is, if we set 
\begin{align}\label{equ:btx}
\scrB_{t,X}:=\bigcup_{|s|\leq X}\widetilde{\scrA}_tU(-s)
=\left\{(\delta v'+sv_n, v'', v_n): (v', v'', v_n)\in \scrC^+, \, 0\leq \delta\leq e^{-nt},\, |s|\leq X\right\},
\end{align}
then
\begin{align}\label{scrEptformula}
\scrE_t'=\left\{\Gamma g\in \scrX: {\hat{\ZZ}^ng}\cap \scrB_{t,X}=\emptyset\right\}=\scrH(\scrB_{t,X})^c.
\end{align}
Let us set
\begin{align}
\scrB'_{t,X}:=\left\{(\delta v'+sv_n, v'', v_n): (v', v'', v_n)\in \scrC,\, v_n\geq 0,\, 0\leq \delta\leq e^{-nt},\, |s|\leq X\right\}.
\end{align}
Note that $\scrB'_{t,X}$ is compact and contains $\scrB_{t,X}$. Hence 
$\scrH(\scrB'_{t,X})^c\subset \scrX$ is open and is contained in $\scrE'_{t}$. 
Now by Remark \ref{rmk:port} (``(1) $\Rightarrow$ (3)"), we have
\begin{align*}
\liminf_{t\to\infty}\lambda(\varphi_t(\scrE'_{t_0}))&\geq \liminf_{t\to\infty}\lambda(\varphi_t(\scrH(\scrB'_{t_0,X})^c))\geq \mu(\scrH(\scrB'_{t_0,X})^c)).
\end{align*}
Note that $\scrB'_{t,X}$ is decreasing in $t$, thus $\scrH(\scrB'_{t,X})^c$ is increasing in $t$. Since $t_0$ is chosen arbitrarily, we have
\begin{align*}
\liminf_{t\to\infty}\lambda\left(\left\{ x\in\scrX : L^{-n/k} r_1(x,\scrH(L^{-1}\mathcal{C}^+)) > X \right\}\right)&\geq \lim_{t_0\to\infty}\mu(\scrH(\scrB'_{t_0,X})^c)=\mu\left(\bigcup_{t_0>0}\scrH(\scrB'_{t_0,X})^c\right)\\
&=\mu\biggl(\biggl\{\Gamma g\in \scrX: \hat{\ZZ}^ng\cap\Bigl(\bigcap_{t_0>0} \scrB'_{t_0,X}\Bigr)=\emptyset\biggr\}\biggr).
\end{align*}
Here for the last equality we used the uniform boundedness  of $\{\scrB'_{t_0,X}\}_{t_0>0}$ and the fact that
$\scrB'_{t_0,X}$ is decreasing in $t_0$. Now note that 
$$
\bigcap_{t_0>0} \scrB'_{t_0,X}=\left\{(sv_n, v'', v_n): ( v'', v_n)\in \mathrm{pr}(\scrC),\, v_n\geq 0,\, |s|\leq X\right\}.
$$
This set differs from 
\begin{align*}
\scrC(\B_X^k,\scrW)=\left\{(sv_n, v'', v_n): ( v'', v_n)\in \mathrm{pr}(\scrC),\, v_n> 0,\, |s|\leq X\right\}
\end{align*}
by a set of Lebesgue measure zero. Thus we get
\begin{align*}
\liminf_{t\to\infty}\lambda\left(\left\{ x\in\scrX : L^{-n/k} r_1(x,\scrH(L^{-1}\mathcal{C}^+)) > X \right\}\right)&\geq \mu\biggl(\biggl\{\Gamma g\in \scrX: \hat{\ZZ}^ng\cap\scrC(\B_X^k,\scrW)=\emptyset\biggr\}\biggr)=\Psi_0(\B_X^k).
\end{align*}
Combining the upper and lower bounds we get the desired equation in \eqref{equ:cuspapp+}. 
\end{proof}

\section{Tail estimates}\label{tailSEC}

In this section we study the asymptotic estimates for the tails of the limit distribution
\begin{equation}\label{intpsi0eqPsi0}
\int_X^\infty \psi_0(r)\, dr =\Psi_0(\B_X^k)
\end{equation}
of the previous section. 
Recall
\begin{equation}\label{LD01978}
\Psi_0(\B_X^k) =\mu\left(\left\{ \Gamma g \in\scrX :  \hatZZ^n g \cap \scrC(\B_X^k,{{\scrW}}) =\emptyset \right\} \right), 
\end{equation}
where 
\begin{equation}\label{scrCBxB1def}
\scrC(\B_X^k,{{\scrW}}) = \left\{ y \in\RR^n :  |(y_1,\ldots,y_k)| \leq y_n X,\,  (y_{k+1},\ldots, y_n)\in {{\scrW}}  \right\}.
\end{equation}
We prove the following asymptotic estimates for the limit distribution in \eqref{intpsi0eqPsi0}, for both large and small  values of $X$. See also  \cite{Gustav24,MS11,Strombergsson11} for tail asymptotics for lattice void distributions in various general settings.
\begin{thm}\label{thm:asestgen}
Let $k\geq1$, $m\geq0$, $n=k+m+1$
and let $\scrW$ be a Borel subset of $\R^m\times\R_{>0}$ of finite volume, and assume $\scrW$ is bounded if $n=2$. 
We have
\begin{align}\label{equ:genestsmallx}
\int_{0}^X\psi_0(r) dr=\kappa X^k+O(X^{2k}),\quad \forall\, X>0,
\end{align}
with
\begin{align}\label{equ:kappagen}
\kappa=\frac{\vol_{\RR^k}(\B_1^k)}{\zeta(n)}\int_{\scrW}y_n^k\,dy_{k+1}\cdots dy_n.
\end{align}
If we further assume that $\scrW$ is bounded and contains the positive half 
$N\cap\{w_{m+1}>0\}$ of some neighbourhood $N$ of the origin in $\R^{m+1}$, then we have
\begin{align}\label{equ:genestlower}
\int_X^{\infty}\psi_0(r) dr \asymp_{n,\scrW} X^{-(n-1)}
\cdot
\begin{cases}
1&\text{if }\: k\geq2
\\
(\log X)^m &\text{if }\: k=1,
\end{cases}
\qquad\text{as }\: X\to\infty.
\end{align}
\end{thm}

\begin{remark}
It will be clear from our method that the boundedness assumption on $\scrW$ implies the lower bound in \eqref{equ:genestlower}, while the assumption that $\scrW$ contains the positive half $N\cap\{w_{m+1}>0\}$  ensures the upper bound in \eqref{equ:genestlower}. We note that when $k=n-1$, this upper bound already follows from the general bound of Athreya and Margulis \cite[Theorem 2.2]{AthreyaMargulis}. Thus in this case \eqref{equ:genestlower} holds for all bounded $\scrW$. 
\end{remark}

\subsection{Small $X$}

For \eqref{equ:genestsmallx}, we prove the following slightly more general estimate (cf.~\cite[Lemma 5.3]{BG2023} for an alternative proof of \eqref{equ:estsmallx} below). 
\begin{prop}
Let $B\subset \RR^n$ be a finite-volume set satisfying $B\cap (-B)\subset \{0\}$.
If $n\geq 3$ then we have 
\begin{equation}\label{equ:estsmallx}
\frac{\vol_{\RR^n}(B)}{\zeta(n)}-\frac{\vol_{\RR^n}(B)^2}{\zeta(n)^2}\leq \mu\left(
\scrH(B)\right) \leq \frac{\vol_{\RR^n}(B)}{\zeta(n)}.
\end{equation} 
If $n=2$, further assume that $B\subset \mathbb{D}$ with $\mathbb{D}:=\{z\in \RR^2: \|z\|_2<1\}$, then we have
\begin{align}\label{equ:estsmllxn2}
\mu(\scrH(B))=\frac{\vol_{\RR^2}(B)}{\zeta(2)}.
\end{align}
\end{prop}

\begin{proof}
First assume $n\geq 3$.
Set 
$$
{{\scrH}}_B^1:=\left\{\Gamma g\in \scrX: \#(\hatZZ^n g\cap B)=1\right\}\quad\text{and}\quad \scrH_B^{>1}:= \left\{\Gamma g\in \scrX: \#(\hatZZ^n g\cap B)>1\right\}
$$ 
so that $\scrH(B)=\scrH_B^1\sqcup \scrH_B^{>1}$. 
Let $\hat{\chi}_B$ be the primitive Siegel transform of the indicator function $\chi_B$ of $B$. By Siegel's volume formula \cite{Siegel1945} we have
 $$
 \mu(\scrH(B))\leq \int_{\scrX}\hat{\chi}_B d\mu=\frac{\vol_{\RR^n}(B)}{\zeta(n)}.
 $$
 On the other hand, we have 
 \begin{align*}
 \int_{\scrX}\hat{\chi}_B d\mu=\int_{\scrH_B^1}\hat{\chi}_Bd\mu+\int_{\scrH_B^{>1}}\hat{\chi}_Bd\mu=\mu(\scrH_B^1)+\int_{\scrH_B^{>1}}\hat{\chi}_Bd\mu.
 \end{align*}
By assumption $B\cap (-B)\subset \{0\}$, and hence, for any $\Gamma g\in \scrH_B^{>1}$, 
we have that
any two points in $\hatZZ^n g\cap B$ form a linearly independent pair. This shows that if $\hat{\chi}_B(\Gamma g)=\ell\geq 2$, then  $$
\sum_{\substack{(v_1,v_2)\in\hatZZ^n g\times \hatZZ^n g\\ \text{$v_1, v_2$ lin. ind.}}}\chi_B(v_1)\chi_B(v_2)=\ell(\ell-1)\geq \hat{\chi}_B(\Gamma g).
 $$
 This, together with the second moment formula of the primitive Siegel transform \cite{Rogers1955}
proves that
 $$
 \int_{\scrH_B^{>1}}\hat{\chi}_Bd\mu\leq \int\sum_{\substack{(v_1,v_2)\in\hatZZ^n g\times \hatZZ^n g\\ \text{$v_1, v_2$ lin. ind.}}}\chi_B(v_1)\chi_B(v_2)d\mu=\frac{\vol_{\RR^n}(B)^2}{\zeta(n)^2}.
 $$
Hence
 \begin{align*}
 \mu(\scrH(B))&=\mu(\scrH^1_B)+\mu(\scrH_B^{>1})\geq \mu(\scrH_B^1)= \int\hat{\chi}_B d\mu-\int_{\scrH_B^{>1}}\hat{\chi}_Bd\mu\geq \frac{\vol_{\RR^n}(B)}{\zeta(n)}-\frac{\vol_{\RR^n}(B)^2}{\zeta(n)^2},
 \end{align*}
thus yielding the desired estimate.

Now we treat the case when $n=2$. Since any unimodular lattice of $\RR^2$ cannot have two linearly independent points in $\mathbb{D}$, by our assumptions on $B$, for any $x=\Gamma g\in \scrX$, $\#(\hat{\ZZ}^2g\cap B)\leq 1$. This shows that $\hat{\chi}_B=\chi_{\scrH(B)}$, from which \eqref{equ:estsmllxn2} follows directly from Siegel's volume formula \cite{Siegel1945}. 
\end{proof}

\begin{proof}[Proof of \eqref{equ:genestsmallx}]
It suffices to show there exists some $X_0>0$ such that \eqref{equ:genestsmallx} holds for all $0<X<X_0$. Take $B=\scrC(\B_X^k,\scrW)$. Note that $B$ satisfies $B\cap (-B)=\emptyset$. Then if $n\geq 3$,
applying \eqref{equ:estsmallx} to $B$ we get 
\begin{align}\label{asytailcon}
\int_{0}^X\psi_0(r) dr=
\mu(\scrH(B))
=\frac{\vol_{\RR^n}(B)}{\zeta(n)}+O(\vol_{\RR^n}(B)^2)=\kappa X^k+O(X^{2k}),\qquad \forall\, X>0, 
\end{align}
with $\kappa$ as in \eqref{equ:kappagen}. 

If $n=2$, since $\scrW$ is bounded, there exists some $T>0$ such that $\scrW\subset (0,T)$. Let $D_T:=\diag(2T, 1/(2T))\in G$ and let $X_0=\frac{1}{8T^2}$. Then we have $BD_T\subset [0,1/2]^2\subset \mathbb{D}$ for any $0<X<X_0$. We can then apply \eqref{equ:estsmllxn2} to $BD_T$  to get 
$$
\int_0^X\psi_0(r)dr=\mu(\scrH(B))=\mu(\scrH(BD_T))=\frac{\vol_{\RR^2}(BD_T)}{\zeta(2)}=\frac{\vol_{\RR^2}(B)}{\zeta(2)}=\kappa X^2,\qquad \forall\, 0<X<X_0,
$$
with $\kappa$ again as in \eqref{equ:kappagen}. 
\end{proof} 

\subsection{Large $X$}
\label{largeXsec}

We will here prove \eqref{equ:genestlower}. Thus we further assume that $\scrW$ is bounded and contains the positive half 
$N\cap\{w_{m+1}>0\}$ of some neighbourhood $N$ of the origin in $\R^{m+1}$. 
We start the proof 
by making some reductions.
Since the measure $\mu$ on $\scrX$ is invariant under right multiplication by $\Phi(t)$ for all $t$,
it follows from \eqref{LD01978} that
for any $c>0$ and $X>0$,
\begin{align}\label{LD01978transf}
\Psi_0(\B_X^k) =\mu\left(\left\{ \Gamma g \in\scrX :  \hatZZ^n g \cap \scrC(\B_X^k,{{\scrW}})\,\Phi\bigl(-n^{-1}\log(cX)\bigr) 
=\emptyset \right\} \right), 
\end{align}
and here we compute via
\eqref{scrCBxB1def} that
\begin{align}\label{CPhiformula}
\scrC(\B_X^k,\scrW)
\,\Phi\bigl(-n^{-1}\log(cX)\bigr) 
=\left\{ y \in\RR^n :  |(y_1,\ldots,y_k)| \leq c^{-1}y_n,  (y_{k+1},\ldots, y_n)\in (cX)^{\frac kn}\scrW \right \}.
\end{align}
Now let us set
\begin{align}\label{CkmRSdef}
\scrC_{k,m}(R,S):=\bigl\{y\in\RR^n : \|(y_1,\ldots,y_k)\|_2\leq y_n\leq R, \: \| (y_{k+1},\ldots, y_{n-1})\|_2\leq S \bigr\}
\qquad (R,S>0),
\end{align}
where $\|\cdot\|_2$ denotes the standard $\ell^2$-norm on $\RR^d$ for any $d\geq1$.
Note that $\scrC_{k,m}(R,S)$ is a product of 
an $m$-dimensional ball of radius $S$
and a $(k+1)$-dimensional cone of height and radius $R$, with apex at the origin.
(If $m=0$, the definition should be interpreted as saying
$\scrC_{k,0}(R,S)=\{y\in\R^{k+1}:\|(y_1,\ldots,y_k)\|_2\leq y_{k+1}\leq R\}$.)
We also introduce the short-hand notation
\begin{align*}
\scrC_{k,m}(R):=\scrC_{k,m}(R,R),
\end{align*}
and set
\begin{align}\label{FkmRdef}
F_{k,m}(R):=
\mu\left(\left\{ \Gamma g \in\scrX :  \hatZZ^n g \cap \scrC_{k,m}(R)=\emptyset \right\} \right).
\end{align}
It will be important later to note that for any 
$R,S>0$, the matrix
$$
 \begin{pmatrix} (S/R)^{\frac mn}1_k & 0 & 0 \\ 0  &  (S/R)^{-\frac{k+1}n}1_m & 0 \\  0 & 0  & (S/R)^{\frac mn} \end{pmatrix}
$$ 
lies in $G$,
and maps the set
$\scrC_{k,m}(R,S)$ to $\scrC_{k,m}\bigl(R^{\frac{k+1}n}S^{\frac mn}\bigr)$;
therefore
\begin{align}\label{pCkmRSformula}
\mu\left(\left\{ \Gamma g \in\scrX :  \hatZZ^n g \cap \scrC_{k,m}(R,S)=\emptyset \right\} \right)
=F_{k,m}\bigl(R^{\frac{k+1}n}S^{\frac mn}\bigr),
\qquad (R,S>0).
\end{align}

Let us fix positive constants $c_1,\ldots,c_4$ (which only depend on the fixed norm $|\cdot|$ and  the set $\scrW$)
such that
\begin{align}\label{normcomparison}
\scrB_{c_1}^k\subset\B_1^k\subset\scrB_{c_2}^k
\qquad\text{and}\qquad
\scrB_{c_3}^m\times(0,c_3]\subset\scrW\subset\scrB_{c_4}^{m+1},
\end{align}
where $\scrB_r^d:=\{z\in\RR^d:\|z\|_2\leq r\}$ for any $d\geq1$. Note that the existence of $c_3$ and $c_4$ follows from our extra assumptions on $\scrW$. 
Using \eqref{CPhiformula} and \eqref{normcomparison}
one verifies that
\begin{align*}
\scrC(\B_X^k,\scrW)\, & \Phi\bigl(-n^{-1}\log(c_2X)\bigr)
\subset\scrC_{k,m}(c_4c_2^{\frac kn}X^{\frac kn})
\end{align*}
and 
\begin{align*}
\scrC_{k,m}(c_3c_1^{\frac kn}X^{\frac kn})\setminus\{y_n=0\}
\subset \scrC(\B_X^k,\scrW)\,\Phi\bigl(-n^{-1}\log(c_1X)\bigr),
\end{align*}
for all $X>0$.
Using now \eqref{LD01978transf}
and the fact that for $\mu$-almost all points $\Gamma g\in\scrX$
we have
$\hatZZ^n g\cap\{y_n=0\}=\emptyset$,
we conclude that
\begin{align}\label{Phi0bounds1}
\Psi_0(\B_X^k)\geq F_{k,m}\bigl(c_4c_2^{\frac kn}X^{\frac kn}\bigr)
\qquad\text{and}\qquad
\Psi_0(\B_X^k)\leq F_{k,m}\bigl(c_3c_1^{\frac kn}X^{\frac kn}\bigr).
\end{align}
Now \eqref{equ:genestlower} 
is an immediate consequence of \eqref{Phi0bounds1} and the following estimate.
\begin{prop}\label{FkmestimatePROP}
Let $k\geq1$, $m\geq0$, and $n=k+m+1$.
Then for all $R>0$ we have
\begin{align}\label{FkmestimatePROPres}
F_{k,m}(R)\asymp_n (R+1)^{-\frac{n(n-1)}k}
\cdot
\begin{cases}
1&\text{if }\: k\geq2
\\
(\log(R+2))^m &\text{if }\: k=1.
\end{cases}
\end{align}
\end{prop}
To start the proof of Proposition \ref{FkmestimatePROP},
let us note that since $\scrC_{k,m}(R)$ is convex and in particular star-shaped at the origin,
\eqref{FkmRdef} can be rewritten as
\begin{align}\label{FkmRdef2}
F_{k,m}(R):=
\mu\left(\left\{ \Gamma g \in\scrX :  \ZZ^n g \cap \scrC_{k,m}(R)=\{0\} \right\} \right).
\end{align}
Proposition \ref{FkmestimatePROP} will now be proved by
induction over the dimension $m$,
using estimates from \cite{Strombergsson11}. 

First assume $m=0$; then recall that
$\scrC_{k,0}(R)=\{y\in\R^{k+1}:\|(y_1,\ldots,y_k)\|_2\leq y_{k+1}\leq R\}$.
In this case we have
\begin{align}\label{Fk0Rasymp}
F_{k,0}(R)\asymp
\min(1,R^{-(k+1)})\qquad\forall R>0.
\end{align}
(In the following, the implied constant in any ``$\ll$'' and ``$\asymp$'' is allowed to depend on $n$,
and thus on $k$ and $m$.)
Indeed, the lower bound
$F_{k,0}(R)\gg\min(1,R^{-(k+1)})$
follows from 
\cite[Prop.\ 2.9]{Strombergsson11},
since there exists a subset $M\subset\mathrm{S}_1^k$
of positive measure such that
$c\scrC_{k,0}(R)\cap v^\perp=\{0\}$ for all $v\in M$ and all $c,R>0$;
the upper bound
$F_{k,0}(R)\ll\min(1,R^{-(k+1)})$
follows from \cite[Prop.\ 2.8]{Strombergsson11},
or alternatively from \cite[Theorem 2.2]{AthreyaMargulis}.
Note that \eqref{Fk0Rasymp} is equivalent with the estimate \eqref{FkmestimatePROPres}
in the case $m=0$.

From now on we assume $m\geq1$.
We will first bound $F_{k,m}(R)$ from above.
If $R\leq10$ then the fact that $F_{k,m}(R)$ is bounded above by the right hand side of
\eqref{FkmestimatePROPres} is obvious,
since $F_{k,m}(R)\leq1$ while the right hand side is $\gg1$.
Hence from now on we may assume $R\geq10$.

By \cite[Prop.\ 2.8]{Strombergsson11},
there is a constant $c>0$ which only depends on $n$ such that
\begin{align}\label{FkmRlbfromProp2p8}
F_{k,m}(R)\ll 
R^{-n}\int_{\mathrm{S}_1^{n-1}}
p\Bigl(\scrC_{k,m}(cR^{\frac n{n-1}}) 
\cap v^\perp\Bigr)\,dv,  
\end{align}
where $\mathrm{S}_1^{n-1}=\{v\in\RR^n:\|v\|_2=1\}$ and
where for any $v\in\mathrm{S}_1^{n-1}$ and 
any (measurable) subset $\scrD\subset\R^n$ we write
\begin{align}\label{pDvperpDEF}
p(\scrD\cap v^\perp)
:=\mu_0\bigl(\bigl\{\Gamma_0g\in\GamGG : \ZZ^{n-1}g\cap L_v(\scrD)\subset\{0\}\bigr\}\bigr),
\end{align}
with $L_v$ being an arbitrary volume preserving linear space isomorphism
from $v^\perp$ onto $\R^{n-1}$.
(Recall that $\mu_0$ denotes the right $G_0$-invariant probability measure on $\GamGG$, 
with $G_0=\SL(n-1,\RR)$ and $\Gamma_0=\SL(n-1,\ZZ)$.
The right hand side of  \eqref{pDvperpDEF} is independent of the choice of $L_v$
since $\mu_0$ is invariant under right $G_0$-multiplication,
and also invariant under the map $\Gamma_0 g\mapsto\Gamma_0\gamma^{-1}g\gamma$
for any fixed $\gamma\in\GL(n-1,\ZZ)$.)

Given any $v\in\mathrm{S}_1^{n-1}$
we set
\begin{align}\label{w1w2w3NEWDEF}
w_1:=\|(v_1,\ldots,v_k)\|_2;\qquad
w_2:=\|(v_{k+1},\ldots,v_{n-1})\|_2;\qquad
w_3:=|v_{n}|.
\end{align}
In the discussion below, we will always assume that $w_1,w_2,w_3$ are \textit{positive;}
this is permitted since it merely means removing a null set from the integral in \eqref{FkmRlbfromProp2p8}.
To bound the integrand in \eqref{FkmRlbfromProp2p8},
note that for any $R>0$,   
since $\scrC_{k,m}(R)$ is preserved by any orthogonal map of the form
$\begin{pmatrix} K_1&0&0\\0&K_2&0\\0&0&\pm 1\end{pmatrix}$ with $K_1\in\OOO(k)$ and $K_2\in\OOO(m)$,
we have
\begin{align*}
p\bigl(\scrC_{k,m}(R)\cap v^\perp\bigr)=p\bigl(\scrC_{k,m}(R)\cap \tv^\perp\bigr),
\qquad\text{with }\:
\tv=\bigl(\underbrace{w_1,0,\ldots,0}_{k\text{ entries}},\underbrace{w_2,0,\ldots,0}_{m\text{ entries}},w_3\bigr).
\end{align*}
Let $\widetilde{\mathrm{pr}}:\RR^n\to\RR^{n-1}$
be the projection map sending 
$(y_1,\ldots,y_n)$ to $(y_1,\ldots,y_k,y_{k+2},\ldots,y_n)$.
Note that $y\in\RR^n$ lies in $\tv^\perp$ 
if and only if $w_1y_1+w_2y_{k+1}+w_3y_n=0$.
Hence, since we are assuming $w_2>0$, the restriction of 
$\widetilde{\mathrm{pr}}$ to $\tv^\perp$ is a linear bijection
from $\tv^\perp$ onto $\RR^{n-1}$
which scales $(n-1)$-dimensional Lebesgue measure by the factor $w_2$.
It follows that $\widetilde{\mathrm{pr}}\big|_{\tv^\perp}$
composed with dilation by a factor $w_2^{-\frac1{n-1}}$
is a volume preserving linear space isomorphism
from $\tv^\perp$ onto $\R^{n-1}$,
and so
\begin{align}\label{pCkmRcapvperpformula1}
p\bigl(\scrC_{k,m}(R)\cap v^\perp\bigr)
=\mu_0\Bigl(\Bigl\{\Gamma_0g\in\GamGG : \ZZ^{n-1}g\cap 
\widetilde{\mathrm{pr}}\big(\scrC_{k,m}\bigl(w_2^{-\frac1{n-1}} R\bigr)\cap \tv^\perp\bigr)=\{0\}\Bigr\}\Bigr).
\end{align}
Here for any $R'>0$,
\begin{align}\label{projsetformula1}
\widetilde{\mathrm{pr}}\big(\scrC_{k,m}(R')\cap \tv^\perp\bigr)
=\Bigl\{z\in\RR^{n-1} &:\|(z_1,\ldots,z_k)\|_2\leq z_{n-1}\leq R',\:
\\\notag
&\Bigl(\tfrac{w_1}{w_2}z_1+\tfrac{w_3}{w_2}z_{n-1}\Bigr)^2+\| (z_{k+1},\ldots, z_{n-2})\|_2^2\leq {R'}^2\Bigr\}.
\end{align}
(If $m=1$, ``$\| (z_{k+1},\ldots, z_{n-2})\|_2$'' should be interpreted as ``$0$''.)

We will split the domain of integration in \eqref{FkmRlbfromProp2p8} into the three overlapping domains
\begin{align*}
&\scrD_1=\left\{v\in\mathrm{S}_1^{n-1}:w_1\geq \max(2w_2,2w_3)\right\},
\qquad
\scrD_2=\left\{v\in\mathrm{S}_1^{n-1}:2w_2\geq\max(w_1,w_3)\right\},
\\
&\hspace{100pt}
\scrD_3=\left\{v\in\mathrm{S}_1^{n-1}:2w_3\geq\max(w_1,w_2)\right\}.
\end{align*}
Note that $\mathrm{S}_1^{n-1}=\scrD_1\cup\scrD_2\cup\scrD_3$.

\vspace{5pt}

Let us first assume $v\in\scrD_2$, i.e.\ $2w_2\geq\max(w_1,w_3)$.
Then we have, for any $R'>0$,
\begin{align}\label{part1inclusion}
\widetilde{\mathrm{pr}}\big(\scrC_{k,m}(R')\cap \tv^\perp\bigr)
\supset
\scrC_{k,m-1}\bigl(\tfrac15R'\bigr).
\end{align}
Indeed, assume $z\in\scrC_{k,m-1}\bigl(\tfrac15R'\bigr)$,
that is,
$\|(z_1,\ldots,z_k)\|_2\leq z_{n-1}\leq\frac15R'$
and $\|(z_{k+1},\ldots,z_{n-2})\|_2\leq\frac15 R'$.
Then
\begin{align*}
\Bigl|\tfrac{w_1}{w_2}z_1+\tfrac{w_3}{w_2}z_{n-1}\Bigr|
\leq \tfrac{w_1+w_3}{w_2}z_{n-1}
\leq 4z_{n-1}
\leq\tfrac45 R'
\end{align*}
and so
\begin{align*}
\Bigl(\tfrac{w_1}{w_2}z_1+\tfrac{w_3}{w_2}z_{n-1}\Bigr)^2+\| (z_{k+1},\ldots, z_{n-2})\|_2^2
\leq\Bigl(\Bigl(\tfrac45\Bigr)^2+\Bigl(\tfrac15\Bigr)^2\Bigr){R'}^2
<{R'}^2.
\end{align*}
Hence $z\in\widetilde{\mathrm{pr}}\big(\scrC_{k,m}(R')\cap \tv^\perp\bigr)$
(via \eqref{projsetformula1}),
and \eqref{part1inclusion} is proved.

It follows from \eqref{pCkmRcapvperpformula1} and \eqref{part1inclusion} that
\begin{align}\label{part1ineq1}
p\bigl(\scrC_{k,m}(R)\cap v^\perp\bigr)\leq
F_{k,m-1}\bigl(\tfrac15 w_2^{-\frac1{n-1}}R\bigr)
\end{align}
for all $R>0$.
Using also $w_2^{-\frac1{n-1}}\geq1$
and the induction hypothesis
that \eqref{FkmestimatePROPres} holds with $m-1$ in the place of $m$, we get
\begin{align}
p\bigl(\scrC_{k,m}(R)\cap v^\perp\bigr)\ll
(R+1)^{-\frac{(n-1)(n-2)}k}(\log(R+2))^a
\qquad(\forall R>0),
\end{align}
where $a=m-1$ if $k=1$, otherwise $a=0$.
Hence we conclude that
\begin{align}\notag
&R^{-n}\int_{\scrD_2}p\Bigl(\scrC_{k,m}(cR^{\frac n{n-1}})\cap v^\perp\Bigr)\,dv
\ll R^{-\frac{n(n+k-2)}k}(\log R)^a
\qquad(\forall R\geq10).
\end{align}
Since $n+k-2\geq n-1$,
this shows that the $\scrD_2$-part of the bound in
\eqref{FkmRlbfromProp2p8} is bounded above by the right hand side of
\eqref{FkmestimatePROPres}.

\vspace{5pt}

Next we assume $v\in\scrD_3$,
i.e.\ $2w_3\geq\max(w_1,w_2)$.
Then we have, for any $R'>0$,
\begin{align}\label{part2inclusion}
\widetilde{\mathrm{pr}}\big(\scrC_{k,m}(R')\cap \tv^\perp\bigr)
\supset
\scrC_{k,m-1}\Bigl(\tfrac{w_2}{4w_3}R',\tfrac12 R'\Bigr).
\end{align}
Indeed, assume 
$z\in\scrC_{k,m-1}\Bigl(\frac{w_2}{4w_3}R',\frac12 R'\Bigr)$,
that is,
$\|(z_1,\ldots,z_k)\|_2\leq z_{n-1}\leq\frac{w_2}{4w_3}R'\leq\frac12R'$
and $\|(z_{k+1},\ldots,z_{n-2})\|_2\leq\frac12 R'$.
Then 
\begin{align*}
\Bigl|\tfrac{w_1}{w_2}z_1+\tfrac{w_3}{w_2}z_{n-1}\Bigr|
\leq \tfrac{w_1+w_3}{w_2}z_{n-1}
\leq\tfrac{3w_3}{w_2}z_{n-1}
\leq\tfrac{3}{4}R'
\end{align*}
and so
\begin{align*}
\Bigl(\tfrac{w_1}{w_2}z_1+\tfrac{w_3}{w_2}z_{n-1}\Bigr)^2+\| (z_{k+1},\ldots, z_{n-2})\|_2^2
\leq\Bigl(\Bigl(\tfrac34\Bigr)^2+\Bigl(\tfrac12\Bigr)^2\Bigr){R'}^2
<{R'}^2.
\end{align*}
Hence $z\in\widetilde{\mathrm{pr}}\big(\scrC_{k,m}(R')\cap \tv^\perp\bigr)$
and \eqref{part2inclusion} is proved.

It follows from \eqref{pCkmRcapvperpformula1}, \eqref{part2inclusion} and \eqref{pCkmRSformula} that
\begin{align}
p\bigl(\scrC_{k,m}(R)\cap v^\perp\bigr)\leq
F_{k,m-1}\Biggl(\Bigl(\tfrac{w_2}{4w_3}\Bigr)^{\tfrac{k+1}{n-1}}\Bigl(\tfrac12\Bigr)^{\frac {m-1}{n-1}}
w_2^{-\frac1{n-1}}R\Biggr)
\end{align}
for all $R>0$.
Noticing also that $w_3\asymp 1$ (since $2w_3\geq\max(w_1,w_2)$ and $w_1^2+w_2^2+w_3^2=1$),
and using the induction hypothesis
that \eqref{FkmestimatePROPres} holds with $m-1$ in the place of $m$, we get
\begin{align}\label{partD3disc1}
p\bigl(\scrC_{k,m}(R)\cap v^\perp\bigr)\ll
\Bigl(w_2^{\frac k{n-1}}R+1\Bigr)^{-\frac{(n-1)(n-2)}k}\cdot
\Bigl(\log\Bigl(w_2^{\frac k{n-1}}R+2\Bigr)\Bigr)^a
\qquad(\forall R>0),
\end{align}
where $a=m-1$ if $k=1$, otherwise $a=0$.
Using this bound, we get for all $R\geq10$:
\begin{align}\notag
&R^{-n}\int_{\scrD_3}p\Bigl(\scrC_{k,m}(cR^{\frac n{n-1}})\cap v^\perp\Bigr)\,dv
\\\notag
&\ll R^{-n}\int_0^{2/\sqrt5}\int_0^{2/\sqrt5}
\Bigl(w_2^{\frac k{n-1}}R^{\frac n{n-1}}+1\Bigr)^{-\frac{(n-1)(n-2)}k}
\Bigl(\log\Bigl(w_2^{\frac k{n-1}}R^{\frac n{n-1}}+2\Bigr)\Bigr)^a
\, w_1^{k-1}\,dw_1\, w_2^{m-1}\,dw_2
\\\notag 
&\ll
R^{-n}\int_0^1
\Bigl(w_2^{\frac k{n-1}}R^{\frac n{n-1}}+1\Bigr)^{-\frac{(n-1)(n-2)}k}
\Bigl(\log\Bigl(w_2^{\frac k{n-1}}R^{\frac n{n-1}}+2\Bigr)\Bigr)^a\, w_2^{m-1}\,dw_2
\\\notag
&\ll
R^{-n}\int_0^{R^{-n/k}}w^{m-1}\,dw
+R^{-n-\frac{n(n-2)}k}\int_{R^{-n/k}}^1 
\Bigl(\log\Bigl(w^{\frac k{n-1}}R^{\frac n{n-1}}+2\Bigr)\Bigr)^a\,w^{-k}\,dw.
\end{align}
If $k\geq2$ (so that $a=0$), the above expression is seen to be
$\ll R^{-\frac{n(n-1)}k}$.
In the remaining case $k=1$ ($\Rightarrow a=m-1$),
we substitute $w=R^{-n}e^{(n-1)u}$
to get
\begin{align*}
\ll R^{-n(n-1)}+R^{-n(n-1)} &\int_0^{\log(R^{n/(n-1)})} (\log(e^u+2))^{m-1}\,du
\\
&\ll R^{-n(n-1)}\int_0^{\log(R^{n/(n-1)})}(u+1)^{m-1}\,du
\ll R^{-n(n-1)}\bigl(\log R\bigr)^{m}.
\end{align*}
Hence we have proved that (when $R\geq10$) the $\scrD_3$-part
of the bound in \eqref{FkmRlbfromProp2p8} is bounded above by the right hand side of
\eqref{FkmestimatePROPres}.

\vspace{5pt}

It remains to consider points $v\in\scrD_1$, i.e.\ points satisfying
$w_1\geq2\max(w_2,w_3)$.
In this case we will use the following modification of
\eqref{pCkmRcapvperpformula1}, which is proved by the same argument:
\begin{align}\label{pCkmRcapvperpformula2}
p\bigl(\scrC_{k,m}(R)\cap v^\perp\bigr)
=\mu_0\Bigl(\Bigl\{\Gamma_0g\in\GamGG : \ZZ^{n-1}g\cap 
\widehat{\mathrm{pr}}\big(\scrC_{k,m}\bigl(w_1^{-\frac1{n-1}} R\bigr)\cap \tv^\perp\bigr)=\{0\}\Bigr\}\Bigr),
\end{align}
where $\widehat{\mathrm{pr}}:\RR^n\to\RR^{n-1}$
sends $(y_1,\ldots,y_n)$ to $(y_2,\ldots,y_n)$.
We have for any $R'>0$:
\begin{align}\label{projsetformula2}
\widehat{\mathrm{pr}}\big(\scrC_{k,m}(R')\cap \tv^\perp\bigr)
=\Biggl\{z\in\RR^{n-1} :
\sqrt{\Bigl(\tfrac{w_2}{w_1}z_{k}+\tfrac{w_3}{w_1}z_{n-1}\Bigr)^2
+\|(z_1,\ldots,z_{k-1})\|^2_2}\leq z_{n-1}\leq R',\:
\\\notag
\| (z_{k},\ldots, z_{n-2})\|_2\leq R'\Biggr\}.
\end{align}
(If $k=1$, ``$\| (z_1,\ldots, z_{k-1})\|_2$'' should be interpreted as ``$0$''.)

We now claim that
\begin{align}\label{hprCkmRpcapvperpinclusion}
&\widehat{\mathrm{pr}}\big(\scrC_{k,m}(R')\cap \tv^\perp\bigr)
\supset\scrC_{k,m-1}\bigl(\tfrac14R'\bigr)\,\bD\,\bU,
\end{align}
where $\bD=\diag\bigl[2^{-\frac1{k+1}}I_k,I_{m-1},2^{\frac k{k+1}}\bigr]\in G_0$,
and where $\bU=(\bu_{ij})\in G_0$ is the unipotent matrix which has
$\bu_{k,n-1}=\frac{w_2w_3}{w_1^2-w_3^2}$, 
$1$s along the diagonal
(i.e.\ $\bu_{j,j}=1$ for all $j$),
and entries $\bu_{i,j}=0$ for all other pairs $i,j$.
To prove \eqref{hprCkmRpcapvperpinclusion},
let $p$ be an arbitrary point in $\scrC_{k,m-1}\bigl(\tfrac14R'\bigr)$,
and set $v=p\bD$ and $z=p\bD\bU=v\bU$.
Then
\begin{align}\label{hprCkmRpcapvperpinclusionpf4}
\|(v_1,\ldots,v_k)\|_2\leq\tfrac12 v_{n-1}\leq 2^{\frac k{k+1}-3}R'<\tfrac14R',
\qquad 
\|(v_{k+1},\ldots,v_{n-2})\|_2\leq\tfrac14R',
\end{align}
and 
\begin{align}\label{unipotentsubst}
z_{n-1}:=v_{n-1}+\tfrac{w_2w_3}{w_1^2-w_3^2}v_{k}
\qquad\text{and}\qquad z_j=v_j\:\text{ for all }j<n-1;
\end{align}
and our task is to prove that $z\in \widehat{\mathrm{pr}}\big(\scrC_{k,m}(R')\cap \tv^\perp\bigr)$.
It follows from \eqref{hprCkmRpcapvperpinclusionpf4} that
$\|(z_k,\ldots,z_{n-2})\|_2^2=|v_k|^2+\|(v_{k+1},\ldots,v_{n-2})\|_2^2\leq(\frac14R')^2+(\frac14R')^2<{R'}^2$.
Also, since $v\in\scrD_1$,
\begin{align}\label{hprCkmRpcapvperpinclusionpf3}
\frac{w_2w_3}{w_1^2-w_3^2}=\frac{w_2w_3/w_1^2}{1-(w_3/w_1)^2}\leq\frac{4^{-1}}{1-4^{-1}}=\frac13,
\end{align}
and so
$z_{n-1}\leq v_{n-1}+\frac1{3}|v_k|
\leq\frac12R'+\frac1{3}\cdot\frac14R'<R'$.
Hence by \eqref{projsetformula2},
it only remains to prove that
$\sqrt{\Bigl(\tfrac{w_2}{w_1}z_{k}+\tfrac{w_3}{w_1}z_{n-1}\Bigr)^2
+\|(z_1,\ldots,z_{k-1})\|_2^2}\leq z_{n-1}$.
Note that \eqref{hprCkmRpcapvperpinclusionpf4} implies that 
$|v_k|\leq\frac12 v_{n-1}$,
and so $z_{n-1}\geq v_{n-1}-\frac1{3}\cdot\frac12 v_{n-1}\geq0$;
hence it suffices to prove that
\begin{align}\label{hprCkmRpcapvperpinclusionpf2}
\Bigl(\tfrac{w_2}{w_1}z_{k}+\tfrac{w_3}{w_1}z_{n-1}\Bigr)^2
+\|(z_1,\ldots,z_{k-1})\|_2^2\leq z_{n-1}^2,
\end{align}
which is equivalent with
\begin{align}\label{hprCkmRpcapvperpinclusionpf1}
\tfrac{w_2^2}{w_1^2-w_3^2}v_k^2
+\|(v_1,\ldots,v_{k-1})\|_2^2\leq \tfrac{w_1^2-w_3^2}{w_1^2}v_{n-1}^2.
\end{align}
Now since $\tfrac{w_2^2}{w_1^2-w_3^2}\leq\frac13$
(as in \eqref{hprCkmRpcapvperpinclusionpf3})
and $\tfrac{w_1^2-w_3^2}{w_1^2}\geq\frac34$,
\eqref{hprCkmRpcapvperpinclusionpf1} follows from 
the first inequality in \eqref{hprCkmRpcapvperpinclusionpf4}. 
This completes the proof of the inclusion in \eqref{hprCkmRpcapvperpinclusion}.

Using \eqref{pCkmRcapvperpformula2} and \eqref{hprCkmRpcapvperpinclusion}
we conclude that, for any $R>0$,
\begin{align*}
p\bigl(\scrC_{k,m}(R)\cap v^\perp\bigr)\leq F_{k,m-1}\bigl(\tfrac14w_1^{-\frac1{n-1}}R\bigr).
\end{align*}
Using this, we may argue exactly as below \eqref{part1ineq1}
to conclude that the $\scrD_1$-part of the bound in
\eqref{FkmRlbfromProp2p8} is bounded above by the right hand side of
\eqref{FkmestimatePROPres}.
Recall that we have already reached the same conclusion for the $\scrD_2$-part
and the $\scrD_3$-part; hence we conclude that 
$F_{k,m}(R)$ is bounded from above by the right hand side of
\eqref{FkmestimatePROPres}.

\vspace{5pt}

It remains to prove the corresponding \textit{bound from below.} 
By \cite[Prop.\ 2.9]{Strombergsson11},
there is a constant $c>0$ 
(not same as in \eqref{FkmRlbfromProp2p8})
which only depends on $n$ such that
\begin{align}\label{FkmRlbfromProp2p9}
F_{k,m}(R)\gg_n\min\Biggl(1,
R^{-n}\int_{\mathrm{S}_1^{n-1}}
p\Bigl(\scrC_{k,m}(cR^{\frac n{n-1}}) 
\cap v^\perp\Bigr)\,dv\Biggr).
\end{align}
We will here consider the integral restricted to the set: 
\begin{align*}
\scrD_3'=\bigl\{v\in\mathrm{S}_1^{n-1}:w_1,w_2\in\bigl(0,\tfrac13\bigr)\bigr\}.
\end{align*}
Note that $w_3=\sqrt{1-w_1^2-w_2^2}>\sqrt{\frac79}>\frac23$ for all $v\in\scrD_3'$
(in particular $\scrD_3'$ is a subset of $\scrD_3$).
Assuming from now on that $v\in\scrD_3'$,
it follows that for every point $z\in\widetilde{\mathrm{pr}}\big(\scrC_{k,m}(R')\cap \tv^\perp\bigr)$
we have 
(cf.\ \eqref{projsetformula1})
$|z_1|\leq z_{n-1}$ and 
$\bigl|\tfrac{w_1}{w_2}z_1+\tfrac{w_3}{w_2}z_{n-1}\bigr|
\geq\frac{w_3}{2w_2}z_{n-1}$ ($\geq0$),
and so
$z_{n-1}\leq \frac{2w_2}{w_3}R'$.
Hence,
by comparing with \eqref{CkmRSdef},
\begin{align}\label{prCkmRpcontainment}
\widetilde{\mathrm{pr}}\big(\scrC_{k,m}(R')\cap \tv^\perp\bigr)
\subset\scrC_{k,m-1}\biggl(\tfrac{2w_2}{w_3}R',R'\biggr).
\end{align}
Using \eqref{pCkmRcapvperpformula1}, \eqref{prCkmRpcontainment}
and \eqref{pCkmRSformula}, we conclude that,
\begin{align}
p\bigl(\scrC_{k,m}(R)\cap v^\perp\bigr)
\geq F_{k,m-1}\biggl(
\Bigl(\tfrac{2w_2}{w_3}\Bigr)^{\frac{k+1}{n-1}}w_2^{-\frac1{n-1}}R\biggr)
\qquad(\forall R>0).
\end{align}
Recalling that $\frac23<w_3<1$,
and using the induction hypothesis that
\eqref{FkmestimatePROPres} 
is valid for $m-1$,
we obtain:
\begin{align}\label{lowerboundpf2}
p\bigl(\scrC_{k,m}(R)\cap v^\perp\bigr)\gg
\Bigl(w_2^{\frac k{n-1}}R+1\Bigr)^{-\frac{(n-1)(n-2)}k}\cdot
\Bigl(\log\Bigl(w_2^{\frac k{n-1}}R+2\Bigr)\Bigr)^a
\qquad(\forall R>0),
\end{align}
where $a=m-1$ if $k=1$, otherwise $a=0$.
Note that this is exactly the same bound as in \eqref{partD3disc1},
although it is now a bound \textit{from below,}
valid on the smaller set $\scrD_3'$.
By essentially the same computation as the one carried out below 
\eqref{partD3disc1} (but now working with bounds from below),
we conclude that for all $R$ sufficiently large,
\begin{align}\label{lowerboundpf1}
R^{-n}\int_{\scrD_3'}p\Bigl(\scrC_{k,m}(cR^{\frac n{n-1}}) \cap v^\perp\Bigr)\,dv
\gg R^{-\frac{n(n-1)}k}
\cdot
\begin{cases}
1&\text{if }\: k\geq2
\\
(\log R)^m &\text{if }\: k=1.
\end{cases}
\end{align}
It is also immediate from \eqref{lowerboundpf2} that
the left hand side of \eqref{lowerboundpf1} is $\gg1$ uniformly over all $R$ in any fixed bounded interval.
Using these facts in \eqref{FkmRlbfromProp2p9}
we conclude that $F_{k,m}(R)$ is bounded from below by the right hand side of
\eqref{FkmestimatePROPres}.
This concludes the proof of Proposition \ref{FkmestimatePROP}.
\hfill$\square$

\section{Extreme value laws}\label{secX}

We now show that the extreme value law stated in Theorem \ref{mainthm} is an immediate consequence of the limit theorems for hitting times, in particular Corollary \ref{cor3app} and the relevant tail asymptotics. The following theorem establishes the connection between the extreme value law and the hitting time distribution $\psi_0$, which holds in greater generality; cf. \cite{DolgopyatFayadLiu2022,FFT}. 

 \begin{thm}\label{mainthm123}
Let $\lambda$ be an admissible probability measure on $\scrX$.
Then, for any $Y\in\RR$, 
\begin{equation}\label{maineq0}
\lim_{T \to \infty} \lambda\left\{ x\in\scrX :  \sup_{|s|\leq T}  \log\alpha(h_s(x))  >  Y + \frac{k}{n}\, \log T  \right\} 
=  \int_Y^\infty \eta(y) dy ,
\end{equation}
where $\eta\in\L^1(\RR)$ is the probability density
\begin{equation}\label{equ:eta}
\eta(y) = \frac{n}{k}\,  \exp\bigg(-\frac{n}{k} y\bigg)\, \psi_0\left(\exp\bigg(-\frac{n}{k} y\bigg)\right),
\end{equation}
with $\psi_{0}\in \L^1(\RR_{>0})$ the probability density function given in \eqref{equ:probdens} with respect to sections $\scrH(\scrW,L)$ with $\scrW=\mathrm{pr}(\BB_1^{n,+})$. 
\end{thm}
\begin{proof}
Let $\scrC=\BB_1^n$. Note that for any $x\in \scrX$, 
\begin{align*}
\sup_{|s|\leq T}  \log\alpha(h_s(x))  <  \log L
\end{align*}
if and only if
\begin{align*}
r_1(x,\scrH(L^{-1}\scrC)) > T. 
\end{align*}
Hence, by Corollary \ref{cor3app} with $T=X L^{n/k}$, $X=\e^{-\frac{n}{k} Y}$ and $\scrW=\mathrm{pr}(\BB_1^{n,+})$ we have 
\begin{equation}\label{maineq00}
\lim_{T \to \infty} \lambda\left\{ x\in\scrX :  \sup_{|s|\leq T}  \log\alpha(h_s(x))  \geq   Y + \frac{k}{n}\, \log T  \right\} 
=  \int_0^{\exp(-\frac{n}{k} Y)} \psi_0(r) dr.
\end{equation}
The claim \eqref{maineq0} then follows by noting that the above right hand side is continuous in $Y$. 
\end{proof}
To complete the proof of Theorem \ref{mainthm}, note that the tail bounds \eqref{asytail0} and \eqref{asytail1} follow from \eqref{equ:genestsmallx} and \eqref{equ:genestlower} respectively by taking $\scrW=\mathrm{pr}(\BB_1^{n,+})$.

\subsection{Multidimensional logarithm laws}
\label{sec:exelog}
As promised in Remark \ref{rmk:loglaw}, we show in this section that  the logarithm law \eqref{uloglaw} follows from our extreme value laws as stated in Theorems \ref{mainthm} and \ref{mainthm123}.

\begin{cor}
For $\mu$-almost every $x\in \scrX$,
\begin{equation}\label{uloglaw1}
\limsup_{T\to\infty}\sup_{|s|\leq T}  \frac{\log\alpha(h_s(x))}{\log T}  = \frac{k}{n}.
\end{equation}
\end{cor}

\begin{proof}
The upper bound is standard and follows from the Borel-Cantelli lemma. For completeness we give a short proof here.  First note that it is equivalent to show for $\mu$-almost every $x\in \scrX$,
$$
\limsup_{|s|\to\infty}  \frac{\log\alpha(h_s(x))}{\log |s|}  \leq  \frac{k}{n}.
$$
Moreover,
$$
\log\alpha(h_s(x))=\log\alpha(h_{\left \lfloor{s}\right \rfloor}(x))+O(1), \quad \forall\, s\in \RR^k\ \text{and}\ x\in \scrX,
$$
where $\left \lfloor{s}\right \rfloor\in\ZZ^k$ denotes the unique integer point in the cube 
$s+(-\frac12,\frac12]^k$.
Using also the continuity of the measure $\mu$ on a decreasing sequence of sets, it suffices to show for every $\epsilon>0$, the following set 
\begin{align*}
\scrL_{\epsilon}:=\left\{x\in \scrX: \limsup_{\substack{s\in \ZZ^k\\ |s|\to\infty}}\frac{\log \alpha(h_s(x))}{\log |s|}\leq \frac{k}{n}+\epsilon\right\}
\end{align*}
has full $\mu$-measure, or equivalently, $\mu(\scrL_{\epsilon}^c)=0$. For each $s\in \ZZ^k\setminus\{0\}$, let 
$$
B_s:=\left\{x\in \scrX: \log\alpha(x)>(\tfrac{k}{n}+\tfrac{\epsilon}{2})\log |s|\right\}.
$$ 
Note that
$
\scrL_{\epsilon}^c\subset \limsup_{\substack{s\in \ZZ^k\\|s|\to\infty}}h_{-s}B_s$ and $B_s\subset \scrH(\BB_{r_s}^n)$ with $r_s=|s|^{-(\frac{k}{n}+\frac{\epsilon}{2})}$. Thus by Siegel's volume formula \cite{Siegel1945} we have
$$
\mu(B_s)\ll \vol_{\RR^n}(\BB_{r_s}^n)\asymp |s|^{-(k+\frac{n\epsilon}{2})},\quad \forall\, s\in \ZZ^k\setminus\{0\},
$$ 
implying that $\sum_{s\in \ZZ^k\setminus \{0\}}\mu(B_s) <\infty$. Then by the Borel-Cantelli lemma and $h_s$-invariance of the Haar measure we have $\mu(\limsup_{\substack{s\in \ZZ^k\\|s|\to\infty}}h_{-s}B_s)=0$, hence also $\mu(\scrL_{\epsilon}^c)=0$.


For the lower bound, we note that the extreme value law implies that
\begin{equation}
\lim_{T\to\infty}\sup_{|s|\leq T}  \frac{\log\alpha(h_s(x))}{\log T}  = \frac{k}{n}
\end{equation}
in probability, for $x$ distributed according to $\mu$.
 Hence, there exists a subsequence $\{T_j\}_j$ such that, for $\mu$-almost every $x$,
\begin{equation}
\lim_{j\to\infty}\sup_{|s|\leq T_j}  \frac{\log\alpha(h_s(x))}{\log T_j}  = \frac{k}{n} ,
\end{equation}
which yields an almost sure lower bound for the lim sup in \eqref{uloglaw1}.
\end{proof}

\subsection{Conjugates}\label{sec:conj}
The results in this paper extend to unipotent actions conjugate to $\{h_s\}_{s\in \RR^k}$. Indeed,   for any $g_0\in G$ let 
$$
h_s^{g_0}(x)=xg_0U(s)g_0^{-1} \qquad (x\in \scrX),
$$
with $U(s)$ as in \eqref{equ:unipotentact}. Define the corresponding sections by $\scrH^{g_0}(\scrW,L):=\scrH(\scrW,L)g_0^{-1}$ with $\scrH(\scrW,L)$ as in \eqref{sect}. Define also 
$$
\Sigma^{g_0}(\scrW,L):=\Sigma(\scrW,L)g_0^{-1}\quad\text{ and}\quad \varphi^{g_0}_t(x)
:=xg_0\Phi(t)g_0^{-1}
$$ 
with $\Sigma(\scrW,L)$ and $\Phi(t)$ as in \eqref{SigmaWLdef} and \eqref{equ:phit} respectively.  Define 
\begin{align}\label{def:hittimeconj}
\xi^{g_0}_j(x,L):=\xi_j(xg_0,L)\in \RR^k,\quad \omega_j^{g_0}(x,L):=\omega_j(xg_0, L)g_0^{-1}\in \Sigma^{g_0}(\scrW,1),\quad \forall\, x\in \scrX,\, j\geq 1,\, L>0.
\end{align}
Note that 
$$
h^{g_0}_s(x)\in \scrH^{g_0}(\scrW,L)\quad \Leftrightarrow\quad h_s(xg_0)\in \scrH(\scrW,L).
$$ 
From this equivalence relation one verifies that $(\xi^{g_0}_j(x,L),\omega_j^{g_0}(x,L))_j$ indeed give all the hitting times  and impact parameters for which 
$$
h^{g_0}_{\xi_j^{g_0}(x,L)}(x)=\varphi^{g_0}_{\log L^{1/k}}(\Gamma \omega_j^{g_0}(x,L))
\in\scrH^{g_0}(\scrW,L),
$$ 
 counted with the correct multiplicity. 

Now by \eqref{def:hittimeconj} we can extend our previous results on hitting time and impact statistics for $\{h_s\}_{s\in \RR^k}$ to similar results for $\{h_s^{g_0}\}_{s\in \RR^k}$. 
To state these results we also need to introduce the following more general notion of admissibility for measures. 
Let $\lambda$ be a Borel probability measure on $\scrX$. We say $\lambda$ is  \textit{admissible with respect to $\varphi_t^{g_0}$} if
$$
\lim_{t\to\infty} \lambda\left(\varphi^{g_0}_t\scrE\right)  = \mu\left(\scrE\right)
$$
holds for every Borel subset $\scrE\subset\scrX$ with $\mu(\partial\scrE)=0$. 
If $\varphi_t^{g_0}=\varphi_t$, this notion agrees with the  admissible measures defined in Definition \ref{def:admi}.
In what follows, we always assume $\lambda$ to be admissible with respect to $\varphi_t^{g_0}$. 
Note that $(g_0)_*\lambda$ is admissible in the sense of Definition \ref{def:admi}, where $(g_0)_*\lambda$ 
is the pushforward measure of $\lambda$ under the right action of $g_0$ on $\scrX$. Indeed, note that $\mu(\partial\scrE)=0$ if and only if $\mu(\partial(\scrE g_0^{-1}))=0$; thus if $\mu(\partial\scrE)=0$, the admissibility assumption on $\lambda$  implies that
$$
\lim_{t\to\infty}(g_0)_*\lambda(\varphi_t\scrE)=\lim_{t\to\infty}\lambda(\scrE\Phi(t)g_0^{-1})=\lim_{t\to\infty}\lambda(\varphi_t^{g_0}(\scrE g_0^{-1}))=\mu(\scrE g_0^{-1})=\mu(\scrE).
$$

The following theorem is the analogue of Theorem \ref{thm2var} and Corollary \ref{cor1} for $\{h_s^{g_0}\}_{s\in \RR^k}$. 

\begin{thm}\label{thm1genconj}
Let $\scrA=\scrA_1\times\cdots\times\scrA_r$ 
where $\scrA_1,\ldots,\scrA_r$ are bounded Borel subsets of $\RR^k$ with
$\vol_{\RR^k}(\partial \scrA_i)=0$, and let $N\in\ZZ_{\geq 0}^r$.
Then
\begin{align*}
\lim_{L\to\infty}
\lambda\left(\left\{ x\in\scrX : \#\left\{ j  : L^{-n/k} \xi^{g_0}_j(x,L)\in\scrA_i \right\}  = N_i\, \forall i\right\}\right)  
= \Psi^{(r)}_N(\scrA) ,
\end{align*}
with $\Psi^{(r)}_N(\scrA)$ as in \eqref{LD019}.
In particular, taking $r=1$, $N=0$ and $\scrA_1=\B_X^k$ we get
\begin{align*}
\lim_{L\to\infty}
\lambda\left(\left\{ x\in\scrX : L^{-n/k} |\xi_1(x,L)|>X \right\}\right)  =  \int_X^\infty \psi_0(r)\, dr,
\end{align*}
with the probability density $\psi_0\in\L^1(\RR_{>0})$ defined by
\begin{align*}
\int_X^\infty \psi_0(r)\, dr = \Psi_0^{(1)}(\B_X^k). 
\end{align*}
\end{thm}
\begin{proof}
By \eqref{def:hittimeconj} we have
\begin{align*}
\left\{ x\in\scrX : \#\left\{ j  : L^{-n/k} \xi^{g_0}_j(x,L)\in\scrA_i \right\}  = N_i\, \forall i\right\}&=\left\{ y\in\scrX : \#\left\{ j  : L^{-n/k} \xi_j(y,L)\in\scrA_i \right\}  = N_i\, \forall i\right\}g_0^{-1}.
\end{align*}
Since $(g_0)_*\lambda$ is admissible with respect to $\varphi_t$, considering $\lim_{L\to\infty}\lambda(\cdots)$ for the above sets and applying Theorem \ref{thm2var} we get the desired result.  The in particular part follows from the same arguments as in the proof of Corollary \ref{cor1}. 
\end{proof}
Similarly, we also have the following counterparts of Theorems \ref{thm3} and \ref{cor1B} for $\{h_s^{g_0}\}_{s\in \RR^k}$.
\begin{thm}\label{thm3gen}
For $\scrA=\scrA_1\times\cdots\times\scrA_r\subset(\RR^k)^r$ bounded with $\vol_{\RR^k}(\partial \scrA_i)=0$,  $\scrD=\scrD_1\times\cdots\times\scrD_r\subset\Sigma^{g_0}(\scrW,1)^r$ with each $\scrD_i g_0$ having compact closure in $HQ$ and $\nu(\partial(\scrD_i g_0))=0$, and $N\in\ZZ_{\geq 0}^r$,
\begin{align*}
\lim_{L\to\infty}
\lambda\left(\left\{ x\in\scrX : \#\left\{ j : L^{-n/k} \xi^{g_0}_j(x,L) \in\scrA_i ,  \; \omega^{g_0}_j(x,L)\in\scrD_i \right\}  = N_i\, \forall i\right\}\right) = \Psi^{(r)}_N(\scrA,\scrD g_0) ,
\end{align*}
with $ \Psi^{(r)}_N$ as in \eqref{LD01910} and $\scrD g_0:=\scrD_1g_0\times \cdots \times \scrD_r g_0\subset \Sigma(\scrW, 1)^r$. 
\end{thm}

\begin{thm}\label{thm:conj}
Let $\scrD=\scrF_H\left\{R(w): w\in \scrB\right\}g_0^{-1}\subset\Sigma^{g_0}(\scrW,1)$ with $\scrB\subset \scrW$ Borel and satisfying $\vol_{\RR^{m+1}}(\partial\scrB)=0$.
Then for any $X>0$, 
\begin{align*}
\lim_{L\to\infty}
\lambda\left(\left\{ x\in\scrX : L^{-n/k} |\xi^{g_0}_1(x,L)|> X ,  \; \omega^{g_0}_1(x,L)\in\scrD \right\}\right)  = \int_{(\B_X^k)^c} \int_{\scrD g_0} \psi(\xi,\omega)\,  d\nu(\omega)\, d\xi,
\end{align*}
with the $\L^1$ probability density $\psi(\xi,\omega)$ as in \eqref{densi}. 
\end{thm}
Next, we can also study the hitting times of $\{h_s^{g_0}\}_{s\in \RR^k}$ to a family of cuspidal neighbourhoods. 
Let $\mathcal{C}\subset \RR^n$ be a compact set with non-empty interior. For any $x\in \scrX$, define 
\begin{align*}
r^{g_0}_1(x, \scrH(\scrC)):=\inf\,\left\{|s|: h^{g_0}_s(x)\in \scrH(\scrC)\right\},
\end{align*}
where $\scrH(\scrC)$ is as in \eqref{equ:hitset}. 
Using the conjugation relation between $h_s$ and $h_s^{g_0}$ one easily verifies that 
\begin{align*}
r_1^{g_0}(x,\scrH(\scrC))=r_1(xg_0, \scrH(\scrC g_0) )\qquad (x\in \scrX),
\end{align*}
where $r_1(xg_0, \scrH(\scrC g_0) )$ is as in \eqref{def:hittime}. 
With this relation and similar arguments as in the proof of Theorem \ref{thm1genconj} we have the following generalisation of Corollary \ref{cor3app} to $\{h_s^{g_0}\}_{s\in \RR^k}$. 
\begin{cor}\label{thm3appgen}
Let $\mathcal{C}\subset \RR^n$ be a compact set with non-empty interior and boundary of Lebesgue measure zero,
satisfying $\scrC=-\scrC$. 
Then for any $X>0$, 
\begin{align*}
\lim_{L\to\infty}
\lambda\left(\left\{ x\in\scrX : L^{-n/k} r^{g_0}_1(x,\scrH(L^{-1}\mathcal{C}))>X \right\}\right)  =  \int_X^\infty \psi_0(r)\, dr,
\end{align*}
with the probability density $\psi_0\in\L^1(\RR_{>0})$  as in Corollary \ref{cor1} and  with $\scrW=\mathrm{pr}(\mathcal{C}^+g_0)$. 
\end{cor}

Now using the same arguments as in the proof of Theorem \ref{mainthm123}, but with Theorem \ref{thm3appgen} in place of Corollary \ref{cor3app} we obtain the following extreme value law for $\{h_s^{g_0}\}_{s\in \RR}$. 

 \begin{thm}\label{mainthm123gen}
For any $Y\in\RR$, 
\begin{align*}
\lim_{T \to \infty} \lambda\left\{ x\in\scrX :  \sup_{|s|\leq T}  \log\alpha(h^{g_0}_s(x))  >  Y + \frac{k}{n}\, \log T  \right\} 
=  \int_Y^\infty \eta(y) dy ,
\end{align*}
where $\eta\in\L^1(\RR)$ is the probability density
\begin{align*}
\eta(y) = \frac{n}{k}\,  \exp\bigg(-\frac{n}{k} y\bigg)\, \psi_0\left(\exp\bigg(-\frac{n}{k} y\bigg)\right),
\end{align*}
with $\psi_{0}\in \L^1(\RR_{>0})$ the probability density function given in \eqref{equ:probdens} with respect to sections $\scrH(\scrW,L)$ with $\scrW=\mathrm{pr}(\BB_1^{n,+}g_0)$. 
\end{thm}
Finally,  since $\BB_1^ng_0$ is the unit ball with respect to the norm $\|v\|_{g_0}:=\|vg_0^{-1}\|$ ($v\in \RR^n$), up to changing the bounding constants the probability density function $\eta$ here also satisfies the asymptotic estimates \eqref{asytail0} and \eqref{asytail1} with the leading coefficient in \eqref{asytail0} replaced by   
\begin{align*}
\kappa=\frac{\vol_{\RR^k}(\B_1^k)}{\zeta(n)}\int_{\mathrm{pr}(\BB^{n,+}_1g_0)}y_n^k\,dy_{k+1}\cdots dy_n.
\end{align*}

\section*{Data access statement}
No new data were generated or analysed during this study.

\end{document}